\documentclass[12pt]{amsart}
\usepackage{amsmath,amssymb}
\usepackage{eufrak}
\usepackage{amsthm}
\usepackage{enumitem}
\usepackage{setspace}
\usepackage{graphicx}
\usepackage{graphics}
\usepackage{color}
\usepackage[dvipsnames]{xcolor}
\usepackage{hyperref}
\usepackage{youngtab}
\usepackage{xcolor}
\usepackage{enumitem}
\usepackage{yfonts}

\hypersetup{ colorlinks,
 linkcolor={blue},  citecolor={BrickRed}, urlcolor={blue}}

\usepackage{amsthm,thmtools,xcolor}

\usepackage{marginnote}
\usepackage{lscape}
\usepackage{faktor}
\usepackage{hyperref}
\usepackage[normalem]{ulem}
\usepackage{bm}
\usepackage{afterpage}
\usepackage{mathrsfs}
\usepackage[all]{xy}
\xyoption{matrix}
\xyoption{arrow}
\usepackage{tikz}
\usetikzlibrary{calc,shadings,patterns}
\usetikzlibrary{matrix}
\usetikzlibrary{arrows}
\usetikzlibrary{matrix}
\usepackage{verbatim}
\usepackage{multirow}

\usetikzlibrary{decorations.pathreplacing}
\usepackage{enumitem,pifont,xcolor}
\definecolor{myblue}{RGB}{0,29,119}

\newtheorem{theorem}{Theorem}[section]
\newtheorem{proposition}[theorem]{Proposition}
\newtheorem{corollary}[theorem]{Corollary}
\newtheorem{lemma}[theorem]{Lemma}
\theoremstyle{definition}
\newtheorem{definition}[theorem]{Definition}

\newtheorem{remark}[theorem]{Remark}
\newtheorem{setup}[theorem]{Set-up}
\newtheorem{remarks}[theorem]{Remarks}

\newtheorem{problem}[theorem]{Problem}

\newtheorem*{theorem*}{Theorem}

\makeatletter
\@addtoreset{equation}{section}
\makeatother

\newcommand{\Zz}{{\mathbb Z}} 

\DeclareMathOperator{\gldim}{gl.\,\! dim}

\DeclareMathOperator{\Hom}{Hom}
\DeclareMathOperator{\ext}{Ext}
\DeclareMathOperator{\Ext}{Ext}

\DeclareMathOperator{\add}{add}

\DeclareMathOperator{\modd}{mod-\!}
\DeclareMathOperator{\Modd}{Mod-\!}
\DeclareMathOperator{\rad}{rad}

\newcommand{\cC}{{\mathcal C}}

\newcommand{\cF}{{\mathcal F}}

\newcommand{\cM}{{\mathcal M}}

\newcommand{\cT}{{\mathcal T}}
\newcommand{\cU}{{\mathcal U}}

\newcommand{\cW}{{\mathcal W}}

\newcommand{\fa}{\large{{\mathfrak a}}}
\newcommand{\fX}{\mathscr{X}}

\newtheorem{innercustomthm}{Theorem}
\newenvironment{customthm}[1]
  {\renewcommand\theinnercustomthm{#1}\innercustomthm}{\endinnercustomthm}
  
\newtheorem{innercustomcor}{Corollary}
\newenvironment{customcor}[1]
  {\renewcommand\theinnercustomcor{#1}\innercustomcor}{\endinnercustomcor}

\newcommand*\circled[1]{\tikz[baseline=(char.base)]{%
            \node[shape=circle,draw,inner sep=2pt] (char) {#1};}}

\newenvironment{axioms}
 {\enumerate[label=\textbf{Axiom\,\arabic*.}, ref=Axiom\,\arabic*]}
 {\endenumerate}
\makeatletter
\newcommand\varitem[1]{\item[\textbf{Axiom\,\arabic{enumi}\rlap{$#1$}.}]%
  \edef\@currentlabel{Axiom\,\arabic{enumi}{$#1$}}}
\makeatother

\providecommand{\AMS}{$\mathcal{A}$\kern-.1667em%
\lower.25em\hbox{$\mathcal{M}$}\kern-.125em$\mathcal{S}$}

\makeatletter
\@namedef{subjclassname@2020}{%
  \textup{2020} Mathematics Subject Classification}
\makeatother

\begin{document}
\title[Construction of $d$-abelian categories via derived categories]{Construction of $d$-abelian categories via derived categories}

\author{Peter J\o rgensen}

\address{Department of Mathematics, Aarhus University, Ny Munkegade 118, 8000 Aarhus C, Denmark}
\email{peter.jorgensen@math.au.dk}

\urladdr{https://sites.google.com/view/peterjorgensen}

\author{Emre Sen}

\address{Iowa City, IA, USA} 
\email{emresen641@gmail.com}

\urladdr{https://sites.google.com/view/emre-sen/home}




%
%

\maketitle

\begin{abstract}
In this work, we provide a simple way to construct $d$-abelian categories via bounded derived categories for certain values of $d$.  Namely, let $\cC$ be an abelian category, and let $\cC[0,m]$ denote the full subcategory of the bounded derived category of $\cC$ whose objects $X$ satisfy that $H_*(X)$ is concentrated in degrees $j$ where $0 \leq j \leq m$. We prove that if $\cC$ is hereditary, then $\cC[0,m]$ is a $d$-abelian category where $d = 3m + 1$.

Beyond offering a uniform method for constructing $d$-abelian categories, this construction allows us to create $d$-abelian categories that exhibit some unexpected pro\-per\-ti\-es depending on the choice of the category $\cC$. For instance, if $\cC$ is the category of abelian groups, then $\cC[0,m]$ is a $d$-abelian category which is not $\mathbb{K}$-linear over a field $\mathbb{K}$ but has set indexed products and coproducts. Similarly, if $\cC$ is the category of coherent sheaves over certain algebraic curves, then $\cC[0,m]$ is a $d$-abelian category without enough injectives.

We extend our results to $(n+2)$-angulated categories. Namely, let $M$ be an $n$-cluster tilting object over an $n$-representation finite algebra and let $\cT$ be the corresponding $(n+2)$-angulated category with $n$-suspension functor $\Sigma_n$.  We prove that the full subcategory $\cT[0,m] = \add \bigoplus^{m}_{j=0}\Sigma^j_n M$ is a $d$-abelian category where $d = (n+2)(m+1)-2$. Furthermore, we show that there is a bijection between the functorially finite wide subcategories of $\add M$ and the functorially finite re\-pe\-ti\-ti\-ve wide subcategories of $\cT[0,m]$.
\end{abstract}
 
\tableofcontents 
 
\section{Introduction} 

In the past decade, $d$-abelian categories, introduced by Jasso in \cite[Def. 3.1]{Jas16}, have provided a framework for axiomatizing $d$-almost split sequences, a concept that naturally arises in higher Auslander--Reiten theory.  Such sequences  play an important role in various areas.  For instance, they were used to study Cohen--Macaulay modules over one dimensional hypersurface singularities in \cite{BIKR08} and rigid Cohen--Macaulay modules over certain rings in \cite{IY08}.  It turns out that $d$-cluster tilting subcategories, as defined by Iyama \cite{iyama2007auslander}, \cite{iyama2011clusterMAINpaper}, are examples of $d$-abelian categories as proved by Jasso in \cite[Thm 3.16]{Jas16}.

In this article, we present a new method for constructing $d$-abelian categories by utilizing full subcategories of derived categories. More precisely, this new approach builds on specific full subcategories of bounded derived categories, which we define in their most general form below.

\begin{definition}
\label{def truncated cat}
Let $\cC$ be an abelian category. For integers $i \leq j$, we define a full subcategory of $D^b(\cC)$, the bounded derived category of $\cC$, as 
\begin{align*}
  \cC[i,j] := \left\{ X \in D^b(\cC) \mid H_{\ast}(X) \text{ is concentrated in degrees } s \text{ with } i \leq s \leq j \right\}.
\end{align*} 
\end{definition}
We recall that $\cC$ is called \emph{hereditary} if $\Ext^i_{\cC}(X,Y) = 0$ for any $X, Y \in \cC$ when $i \geq 2$. We now present our first result.

\begin{customthm}{A}\label{thm hereditary implies d-abelian} Let $\cC$ be a hereditary abelian category. Then, $\cC[0,m]$ is a $(3m+1)$-abelian category. 
\end{customthm}

There are many examples of hereditary abelian categories, among them module categories of hereditary algebras, the category of abelian groups, and the category of sheaves on $\mathbb{P}^1$ or a weighted projective space.  By varying the underlying category $\cC$, we obtain the following results.

\begin{customcor}{B} 
\label{thm MAIN COR}
i) Let $\cC$ be the category of finite dimensional modules over a finite dimensional hereditary algebra. Then, $\cC[0,m]$ is a $(3m+1)$-abelian category.  If the algebra has infinite representation type, then $\cC[0,m]$ has infinitely many indecomposable objects.\vspace{0.1cm}\\
ii) Let $\cC$ be the category of abelian groups. Then, $\cC[0,m]$ is a $(3m+1)$-abelian category which is not $\mathbb{K}$-linear over a field $\mathbb{K}$ but has set indexed products and coproducts. It is $\mathbb{Z}$-linear and it does not have a Serre functor.\vspace{0.1cm}\\
iii) Let $\cC$ be the category of coherent sheaves over the projective line. Then, $\cC[0,m]$ is a $(3m+1)$-abelian category which does not have enough injectives.\vspace{0.1cm}\\
iv) Let $\cC'$ be a category derived equivalent to a hereditary abelian category $\cC$, that is, there is an equivalence of triangulated categories $\cF : D^b(\cC)\xrightarrow{} D^b(\cC')$. Then, $\cF(\cC[0,m])$ is a $(3m+1)$-abelian category.
\end{customcor}



 The study of truncated subcategories of derived categories,
as defined in \ref{def truncated cat}, was initiated in \cite[Def. 1.1, Thm. 1.2]{sen2023higher} for the case where $\cC$ is the module category of a Dynkin quiver, i.e., hereditary and of
finite representation type.  It was shown that the endomorphism algebra of $\cC[0,m]$ in $D^b(\cC)$ is a higher Auslander algebra with global dimension $3m+2$.

A natural question arises for the reverse implication in Theorem \ref{thm hereditary implies d-abelian}.  We prove:

\begin{customthm}{A$'$}\label{thm d-abelian implies hereditary A'}Assume that $m\geq 1$ and $\cC$ is a small abelian category with enough injectives or projectives. If $\cC[0,m]$ is $(3m+1)$-abelian then $\cC$ is a hereditary category.
\end{customthm}

%

Our results are not limited to hereditary abelian categories. We recall that a finite-dimensional $\mathbb{K}$-algebra \(A\) is called \(n\)-representation finite if it possesses an $n$-cluster tilting object and $\gldim A=n$ \cite{iyama2011clusterMAINpaper}. If $A$ is $n$-representation finite, then the \(n\)-cluster tilting object $M$ is unique and 
the subcategory
\[
M[n\mathbb{Z}] := \text{add} \{\, X[\ell n] \mid X \in \add M, \: \ell \in \mathbb{Z} \,\}
\]
becomes an \(n\)-cluster tilting subcategory of \(D^b(A)\). Moreover, it is an \((n+2)\)-angulated category \cite[Thm 1]{geiss2013n}, \cite[Thm 1.23]{iyama2011clusterMAINpaper}.

We prove the following.
\begin{customthm}{C}
\label{thm higher case B, rep finite}
Let $M$ be the $n$-cluster tilting object of $\modd A$ where $A$ is an $n$-representation finite algebra. Then, the full subcategory $\add \bigoplus^{m}_{j=0}M[jn]$ 
of $D^b(\modd A)$ is a $d$-abelian category where $d=(n+2)(m+1)-2$. 
\end{customthm}

Indeed, the previous result is a special case of the following more general set-up. 
\begin{definition}
\label{def truncated subcat in general}
Let $\cT$ be an additive category which is $(n+2)$-angulated with an $n$-suspension functor $\Sigma_n$.  Let $\{\fa_m\}_{m\in\Zz}$ be additive subcategories of $\cT$ satisfying:
\begin{enumerate}[label=\alph*)]
\item $\Hom_{ \cT }(\fa_m,\fa_n)=0$ for $n\notin \{m,m+1\}$,
\item If $t\in\cT$ then $t=a_s\oplus a_{s-1}\oplus\cdots\oplus a_i$ with $a_j\in \fa_j$,
\item $\fa_j=\Sigma^j_n\fa_0$.
\end{enumerate} 
Then, the category $\cT[i,j]:=\add\left(\fa_i,\fa_{i+1},\ldots,\fa_j\right)$ is called a \emph{truncated subcategory of} $\cT$.
\end{definition}


A typical example of a category $\cT$ satisfying all the conditions above is the $(n+2)$-angulated category arising from an $n$-cluster tilting object of an $n$-representation finite algebra. Truncated subcategories of such a category $\cT$ are studied in \cite{sen2023higher} and used to create higher Auslander algebras and new classes of higher representation finite algebras. We want to emphasize that in general truncated subcategories are not extended module categories in the sense of \cite[Def. 2.2]{Zho24}. We prove the following.

\begin{customthm}{D}
\label{thm higher case B most general} 
Let $\cT$ and $\{\fa_m\}_{m\in\mathbb{Z}}$ be as given in Definition \ref{def truncated subcat in general}. Assume further that $\fa_0$ has split idempotents and $\cT[0,m+1]$ is closed under $n$-extensions. Then the truncated subcategory $\cT[0,m]$ of $\cT$ is a $d$-abelian category where $d=(n+2)(m+1)-2$.
\end{customthm}

Our last result concerns certain wide subcategories of $\fa_0$, $\cT$ and $\cT[0,m]$.  Let $A$ be an $n$-representation finite algebra with $n$-cluster tilting object $M$.  Let $\cT = M[n\mathbb{Z}]$ and $\fa_0 = \add M$.  Let $\cW$ be a subcategory of $\fa_0$. We consider the following subcategories
\begin{enumerate}[label=\alph*)]
\item $\underline{\cW}=\add\left(\cW,\Sigma_n\cW,\cdots,\Sigma^m_n\cW\right)$ is a subcategory of $\cT[0,m]$, which we call a repetitive subcategory,
\item $\overline{\cW}:=\add\left(\Sigma^i_n\cW\vert i\in\mathbb{Z} \right)$.
\end{enumerate}

It was proved by Fedele in \cite[Thm. A]{Fe19} that there is a bijection between functorially finite wide subcategories of $\fa_0$ and functorially finite wide subcategories of $\cT$ which generalizes the result for hereditary categories shown by Br\"{u}ning \cite[Thm. 1.1]{Br07}.  We extend this to cover functorially finite repetitive wide subcategories of $\cT[0,m]$, hence obtaining the following result.

\begin{customthm}{E}\label{thm wide subcats} There are inverse bijections between 
\begin{enumerate}[label=\arabic*)]
\item functorially finite wide subcategories $\cW$ of $\fa_0$,
\item functorially finite repetitive wide subcategories $\underline{\cW}$ of $\cT[0,m]$,
\item functorially finite subcategories $\overline{\cW}$ of $\cT$.
\end{enumerate}
\end{customthm}

The paper is organised as follows: We discuss some background material in Section \ref{prelim}.  In Section \ref{section construction of d-abelian category} we give proofs of Theorems \ref{thm hereditary implies d-abelian}, \ref{thm higher case B, rep finite} and \ref{thm higher case B most general} and Corollary \ref{thm MAIN COR}.  In Sections \ref{section d-abelian to actual category} and \ref{section wide} we give proofs of Theorems \ref{thm d-abelian implies hereditary A'} and \ref{thm wide subcats}.  In section \ref{section questions}, we pose some questions.

\subsection{Acknowledgments} This project started at the Auslander Conference 2024. We would like to thank the organizers for their hospitality. This work was supported by a DNRF Chair from the Danish National Research Foundation (grant DNRF156), by a Research Project 2 from the Independent Research Fund Denmark (grant 1026-00050B), and by Aarhus University Research Foundation (grant AUFF-F-2020-7-16).

\section{Preliminaries}\label{prelim} 
\subsection{Higher Auslander--Reiten Theory}
\label{subsec:Higher}

$n$-cluster tilting objects were introduced by Iyama in \cite{iyama2007auslander} where they play a crucial role in higher Auslander--Reiten theory. Let $A$ be a finite dimensional algebra and $\modd A$ be the category of finitely generated $A$-modules. Following \cite{iyama2007auslander} and \cite{iyama2011clusterMAINpaper}, let $\cM$ be a subcategory of $\modd A$. $\cM$ is called $n$-rigid if $\ext^i_{A}(\cM,\cM)=0$ for any $0<i<n$. $\cM$ is called an $n$-cluster tilting subcategory if it is functorially finite and
\begin{align*}
  \cM &= \{\, X\in\modd A\vert \ext^i_{A}(X,\cM)=0 \mbox{ for } 0<i<n \,\} \\
      &=\{\, X\in\modd A\vert \ext^i_{A}(\cM,X)=0 \mbox{ for } 0<i<n \,\}.
\end{align*}
An additive generator of an $n$-cluster tilting subcategory is called an $n$-cluster tilting object.  An algebra $A$ is called $n$-representation finite if it has an $n$-cluster tilting object and the global dimension of $A$ is at most $n$. In this case the \(n\)-cluster tilting object \(M\) is unique and given by:
\[
M := \bigoplus_{j \geq 0} \tau_n^j(DA)
\]
where \(\tau_n := \tau \Omega^{n-1}\) denotes the higher Auslander--Reiten translation and $D=\Hom_{\mathbb{K}}(-,\mathbb{K})$. In this case, the subcategory
\[
M[n\mathbb{Z}] := \text{add} \{\, M[\ell n] \mid \ell \in \mathbb{Z} \,\}
\]
becomes an \(n\)-cluster tilting subcategory of \(D^b(A)\). Moreover, it is an \((n+2)\)-angulated category \cite{geiss2013n}, \cite{iyama2011clusterMAINpaper}.

%

\subsection{$d$-abelian categories}

\begin{definition}\cite[Def. 2.2]{Jas16}
Let $\mathcal{C}$ be an additive category and let $f_0: X_0 \to X_1$ be a morphism in $\mathcal{C}$. An $d$-cokernel of $f_0$ is a sequence
\[
(f_1, \ldots, f_d) : X_1 \to X_2 \to \cdots \to X_{d+1}
\]
such that for all $Y \in \mathcal{C}$, the induced sequence of abelian groups
\[
0 \to \mathcal{C}(X_{d+1}, Y) \to \mathcal{C}(X_d, Y) \to \cdots \to \mathcal{C}(X_1, Y) \to \mathcal{C}(X_0, Y)
\]
is exact. Equivalently, the sequence $(f_1, \ldots, f_d)$ is an $d$-cokernel of $f_0$ if for all $1 \leq k \leq d-1$, the morphism $f_k$ is a weak cokernel of $f_{k-1}$, and $f_d$ is moreover a cokernel of $f_{d-1}$. 

The concept of $d$-kernel of a morphism is defined dually.

A sequence
\[
  X_0 \xrightarrow{ f_0 } X_1 \xrightarrow{} \cdots \xrightarrow{} X_d \xrightarrow{ f_{ d+1 } } X_{d+1}
\]
is called $d$-exact if
\[
  X_1 \xrightarrow{} \cdots \xrightarrow{} X_d \xrightarrow{ f_d } X_{d+1}
\]
is a $d$-cokernel of $f_0$ and
\[
  X_0 \xrightarrow{ f_0 } X_1 \xrightarrow{} \cdots \xrightarrow{} X_d
\]
is a $d$-kernel of $f_d$.
\end{definition}

\begin{definition}\label{def d-abelian}\cite[Def. 3.1]{Jas16}
Let $d$ be a positive integer. A $d$-abelian category is an additive category $\mathcal{M}$ which satisfies the following axioms:
\begin{axioms}
\setcounter{enumi}{-1}
    \item \label{d-ab axiom 0} The category $\mathcal{M}$ is idempotent complete.
    \item \label{d-ab axiom 1} Every morphism in $\mathcal{M}$ has a $d$-kernel and a $d$-cokernel.
    \item \label{d-ab axiom 2} For every monomorphism $f_0: X_0 \to X_1$ in $\mathcal{M}$ and for every $d$-cokernel $(f_1, \ldots, f_n)$ of $f_0$, the following sequence is $d$-exact:
    \[
    X_0 \to X_1 \to \cdots \to X_d \to X_{d+1}
    \]
    
    \varitem{^{op}} \label{d-ab axiom 2 op} For every epimorphism $g_d: X_d \to X_{d+1}$ in $\mathcal{M}$ and for every $d$-kernel $(g_0, \ldots, g_{d-1})$ of $g_d$, the following sequence is $d$-exact:
    \[
    X_0 \to X_1 \to \cdots \to X_d \to X_{d+1}
    \]
\end{axioms}
\end{definition}

\begin{remark} By \cite[Thm. 3.16]{Jas16}, a $d$-cluster tilting subcategory of an abelian category is a $d$-abelian category. The reverse implication is proven in \cite[Thm. 3.20]{Jas16} and \cite[Thm. 7.3]{Kva22}.
\end{remark}

\section{Construction of $d$-abelian categories}\label{section construction of d-abelian category}
In this section, we provide proofs for Theorems \ref{thm hereditary implies d-abelian}, \ref{thm higher case B, rep finite}, and \ref{thm higher case B most general}. The organization of this section is as follows: It consists of five subsections, each beginning with a set-up assumed throughout that subsection.

In detail, we develop the necessary tools to prove Theorems \ref{thm higher case B, rep finite} and \ref{thm higher case B most general} in subsections \ref{subsect set up 1} and \ref{subsect set up 2 d-suspension}, and we prove them in subsection \ref{subsect set up 3 proof of thm C D}. From these results, we deduce Theorem \ref{thm hereditary implies d-abelian} in subsection \ref{subsect set up 4 proof of thm A}. Finally, in subsection \ref{subsect set up 5 proof of Corollary}, we prove Corollary \ref{thm MAIN COR}.

\subsection{Truncation}\label{subsect set up 1}
\begin{setup}\label{setup 1 page 1}
Let $\mathcal{T}$ be an additive category. Let $\{\fa_m\}_{m\in\Zz}$ be additive subcategories of $\cT$ satisfying:
\begin{enumerate}[label=\alph*)]
\item\label{set up 1 a} $\Hom_{ \cT }(\fa_m,\fa_n)=0$ for $n\notin \{m,m+1\}$.
\item\label{set up 1 b} If $t\in\cT$ then $t=a_s\oplus a_{s-1}\oplus\cdots\oplus a_i$ with $a_j\in \fa_j$. Note that $a_j$'s are written in descending order.
\end{enumerate}

We define the following subcategories of $\cT$:
\begin{align}
\label{3.1}
\cT_{\geq i}:=\{a_r\oplus a_{r-1}\oplus\cdots\oplus a_i\vert a_j\in \fa_j\}, \\
\label{3.2}
\cT_{\leq s}:=\{a_s\oplus a_{s-1}\oplus\cdots\oplus a_h\vert a_j\in \fa_j\}, \\
\label{3.3}
\cT_{[i,s]}:=\{a_s\oplus a_{s-1}\oplus\cdots a_{i+1}\oplus a_i\vert a_j\in \fa_j\}.
\end{align}
\end{setup}

\begin{remark} By the set-up a general morphism $\alpha$ in $\cT$ is given by an upper triangular matrix:
\begin{align*}
\begin{pmatrix}
\alpha_{r,r} & \alpha_{r,r-1} & &&\\
&\alpha_{r-1,r-1}&\alpha_{r-1,r-2}&&&\\
&&\alpha_{r-2,r-2}&\ddots&&\\
&&&\ddots&&\\
&&&&\alpha_{h+1,h+1}&\alpha_{h+1,h}\\
&&&&& \alpha_{h,h}
\end{pmatrix}
\end{align*}
Here $\alpha$ is a morphism $a_r\oplus a_{r-1}\oplus \cdots a_h\rightarrow a'_{r}\oplus a'_{r-1}\oplus \cdots a'_h$ where $r\geq h$, and some $a_j$ and $a'_j$ can be zero.
\end{remark}

\begin{lemma}\label{lemma 3 p3} In an additive category assume that $a_1\oplus a_0\cong b_1\oplus b_0$ while 
\begin{align}\label{cond p3}
  \Hom(a_1,b_0)=0, \Hom(b_1,a_0)=0.
\end{align} Then, $a_1\cong b_1$ and $a_0\cong b_0$.
\end{lemma}
\begin{proof}
There are inverse isomorphisms 
\begin{align*}
\xymatrix{a_1\oplus a_0 \ar@<1ex>[rrr]^{ \begin{pmatrix} {\scriptstyle \alpha_{11} } & {\scriptstyle \alpha_{10} } \\ {\scriptstyle 0} & {\scriptstyle \alpha_{00} } \end{pmatrix}}&&& \ar@<1ex>[lll]^{\begin{pmatrix} {\scriptstyle \beta_{11} } & {\scriptstyle \beta_{10} }\\ {\scriptstyle 0} & {\scriptstyle \beta_{00} }\end{pmatrix}
} b_1\oplus b_0}
\end{align*}
which must be given by upper diagonal matrices by Equation \eqref{cond p3}.  Then,
\begin{align*}
\begin{pmatrix} \alpha_{11} &\alpha_{10}\\0&\alpha_{00}\end{pmatrix}\begin{pmatrix} \beta_{11} &\beta_{10}\\0&\beta_{00}\end{pmatrix}=\begin{pmatrix}
\alpha_{11}\beta_{11}&\alpha_{11}\beta_{10}+\alpha_{10}\beta_{00}\\
0&\alpha_{00}\beta_{00}
\end{pmatrix}=\begin{pmatrix}
\operatorname{id}&0\\0&\operatorname{id}
\end{pmatrix},\\[2mm]
\begin{pmatrix} \beta_{11} &\beta_{10}\\0&\beta_{00}\end{pmatrix}\begin{pmatrix} \alpha_{11} &\alpha_{10}\\0&\alpha_{00}\end{pmatrix}=\begin{pmatrix}
\beta_{11}\alpha_{11}&\beta_{11}\alpha_{10}+\beta_{10}\alpha_{00}\\
0&\beta_{00}\alpha_{00}
\end{pmatrix}=\begin{pmatrix}
\operatorname{id}&0\\0&\operatorname{id}
\end{pmatrix}
\end{align*}
so 
\begin{align*}
\xymatrix{a_0 \ar@<1ex>[r]^{ \alpha_{00}}& \ar@<1ex>[l]^{\beta_{00}} b_0},\quad
\xymatrix{a_1 \ar@<1ex>[r]^{ \alpha_{11}}& \ar@<1ex>[l]^{\beta_{11}} b_1}
\end{align*}
are inverse isomorphisms.
\end{proof}

\begin{lemma} \label{lemma 4 p4} Assume $a_q\oplus a_{q-1}\oplus\cdots\oplus a_g\cong b_q\oplus b_{q-1}\oplus\cdots\oplus b_g$ with $a_j,b_j\in \fa_j$ for each $j$. Then, $a_j\cong b_j$ for each $j$.
\end{lemma}
\begin{proof}
We prove by induction on $j$ that 
\begin{align}\label{cond p4}
a_q\oplus\cdots \oplus a_{j+1}\cong b_{q}\oplus\cdots b_{j+1}\quad \text{and}\quad a_j\cong b_j.
\end{align}
Let $j=g$. By our assumptions $a_q\oplus a_{q-1}\oplus\cdots\oplus a_g\cong b_q\oplus b_{q-1}\oplus\cdots\oplus b_g$. By set-up \ref{setup 1 page 1}
\begin{align*}
\Hom_{ \cT }(a_q\oplus a_{q-1}\oplus\cdots\oplus a_{g+1},b_g)=0,\Hom_{ \cT }(b_q\oplus b_{q-1}\oplus\cdots\oplus b_{g+1},a_g)=0.
\end{align*}
Lemma \ref{lemma 3 p3} implies induction at level $g$. Assume induction at level $j$ holds. Then, $a_q\oplus a_{q-1}\oplus\cdots\oplus a_{j+1}\cong b_q\oplus b_{q-1}\oplus\cdots\oplus b_{j+1}$. By set-up \ref{setup 1 page 1},
\begin{align*}
\Hom_{ \cT }(a_q\oplus a_{q-1}\oplus\cdots\oplus a_{j+2},b_{j+1})=0,\Hom_{ \cT }(b_q\oplus b_{q-1}\oplus\cdots\oplus b_{j+2},a_{j+1})=0.
\end{align*}
Now Lemma \ref{lemma 3 p3} implies induction at level $j+1$.
\end{proof}

\begin{lemma}\label{lemma 5 p5} $\cT_{[i,s]}=\cT_{\geq i}\cap\cT_{\leq s}$ for $i\leq s$.
\end{lemma}
\begin{proof}
The inclusion $\cT_{[i,s]}\subseteq\cT_{\geq i}\cap\cT_{\leq s}$ is immediate from Equations \eqref{3.1} through \eqref{3.3}. So we want to prove that $\cT_{[i,s]}\supseteq\cT_{\geq i}\cap\cT_{\leq s}$. If $t\in \cT_{\geq i}\cap\cT_{\leq s}$, then we have
\begin{align*}
t\cong a_r\oplus a_{r-1}\oplus \cdots a_i\quad \text{with}\quad a_j\in \fa_j,\\t\cong b_s\oplus b_{s-1}\oplus \cdots b_h\quad \text{with}\quad b_j\in \fa_j.
\end{align*}
Pick $q\geq\max \{r,s\}$ and $g\leq \min \{i,h\}$. 
\begin{enumerate}[label=\roman*)]
\item If $q>r$ then introduce $a_q,a_{q-1},\ldots,a_{r+1}=0$.
\item\label{p5 1} If $q>s$ then introduce $b_q,b_{q-1},\ldots,b_{s+1}=0$.
\item \label{p5 2} If $g<i$, then introduce $a_{i-1},a_{i-2},\ldots, a_{g}=0$.
\item If $g<h$, then introduce $b_{h-1},b_{h-2},\ldots,b_g=0$.
\end{enumerate}
Then we have
\begin{align*}
t\cong a_q\oplus a_{q-1}\oplus \cdots a_g\quad \text{with}\quad a_j\in \fa_j,\\t\cong b_q\oplus b_{q-1}\oplus \cdots b_g\quad \text{with}\quad b_j\in \fa_j,
\end{align*}
and Lemma \ref{lemma 4 p4} gives $a_j\cong b_j$ for each $j$. But then \ref{p5 1} implies $a_q,a_{q-1},\ldots,a_{s+1}=0$ and combining with \ref{p5 2} shows $t\cong a_s\oplus a_{s-1}\oplus\cdots\oplus a_i\in\cT_{[i,s]}$.
\end{proof}

\begin{lemma} \label{lemma 6 p6} Let $t\in\cT$ and $i\in\Zz$ be given. Write $t=a_q\oplus a_{q-1}\oplus\cdots\oplus a_{g+1}\oplus a_g$ with $a_j\in \fa_j$ and assume $q\geq i\geq g$ by setting any additional $a_j$ which are required equal to zero.\\
The canonical morphism 
\begin{align*}
\xymatrix{a_q\oplus\cdots\oplus a_i\ar[r]^-{\alpha}& a_q\oplus\cdots\oplus a_g=t}
\end{align*} is a strong $\cT_{\geq i}$-cover;
that is, $\Hom_{ \cT }( u,\alpha )$ is bijective for $u$ in $\cT_{ \geq i }$.  Hence the inclusion $\cT_{\geq i}\rightarrow \cT$ has a right adjoint given by $\tau_{\geq i}(a_q\oplus\cdots\oplus a_g)=a_q\oplus\cdots\oplus a_i$.
\end{lemma}
\begin{proof}
$\alpha$ is the inclusion of a summand with complement $a_{i-1}\oplus a_{i-2}\oplus \cdots\oplus a_g$. But if $u\in\cT_{\geq i}$, then 
\begin{align*}
\Hom_{ \cT }(u,a_{i-1}\oplus a_{i-2}\oplus \cdots\oplus a_g)=0
\end{align*}
by set-up \ref{setup 1 page 1}. Hence $\Hom_{ \cT }(u,\alpha)$ is bijective.
\end{proof}

\begin{lemma} \label{lemma 7 p7} Let $t\in\cT$ and $s\in\Zz$ be given. Write $t=a_q\oplus a_{q-1}\oplus\cdots\oplus a_{g+1}\oplus a_g$ with $a_j\in \fa_j$ and assume $q\geq s\geq g$ by setting any additional $a_j$ which are required equal to zero.\\
The canonical morphism 
\begin{align*}
  t = a_q \oplus \cdots \oplus a_g \xrightarrow{ \beta } a_s \oplus \cdots \oplus a_g
\end{align*} is a strong $\cT_{\leq s}$-envelope; that is, $\Hom_{ \cT }( \beta,v )$ is bijective for $v$ in $\cT_{ \leq s }$.  Hence the inclusion $\cT_{\leq s}\rightarrow \cT$ has a left adjoint given by $\tau_{\leq s}(a_q\oplus\cdots\oplus a_g)=a_s \oplus \cdots \oplus a_g$.
\end{lemma}
\begin{proof}
$\beta$ is the projection on a summand with complement $a_{q}\oplus a_{q-1}\oplus \cdots\oplus a_{s+1}$. But if $v\in\cT_{\leq s}$, then 
\begin{align*}
\Hom_{ \cT }(a_{q}\oplus a_{q-1}\oplus \cdots\oplus a_{s+1},v)=0
\end{align*}
by set-up \ref{setup 1 page 1}. Hence $\Hom_{ \cT }(\beta,v)$ is bijective.
\end{proof}

\subsection{Some useful $(d+2)$-angles}\label{subsect set up 2 d-suspension}
\begin{setup}\label{set up p 8}
In a $(d+2)$-angulated category with $d$-suspension $\Sigma_d$, we consider the $(d+2)$-angle
\begin{align*}
\xymatrix{a^0\oplus g^0\ar[r]^-{\partial^0}&a^1\oplus g^1\ar[r]^-{\partial^1}&a^2\oplus g^2\ar[r]^-{\partial^2}&\cdots\ar[r]^-{\partial^d}&a^{d+1}\oplus g^{d+1}\ar[r]^-{\partial^{d+1}}&\Sigma_da^{0}\oplus\Sigma_d g^0}
\end{align*}
where 
\begin{align*}
\partial^0=\begin{pmatrix}
\alpha^0 & 0\\ \varphi^0&\gamma^0 \end{pmatrix},
\partial^1=\begin{pmatrix}
\alpha^1 & 0\\ \varphi^1&\gamma^1\end{pmatrix},
\partial^2=\begin{pmatrix}
\alpha^2 & 0\\ \varphi^2&\gamma^2\end{pmatrix},
\ldots,\partial^d=\begin{pmatrix}
\alpha^d & 0\\ \varphi^d&\gamma^d\end{pmatrix} 
\end{align*}
and $\partial^{d+1}=\begin{pmatrix}
\alpha^{d+1}&0\\0&0
\end{pmatrix}$.\\
We assume
\begin{align}
\Hom(g^{i+1},a^i)=0 \text{ for }\, i=0,\ldots,d \quad \mbox{and} \quad \Hom(g^0,\Sigma^{-1}_da^{d+1})=0.
\end{align}
The angle can be visualized as follows:
\begin{align*}
\xymatrix{g^0\ar[r]^{\gamma^0}&g^1\ar[r]^{\gamma^1}&g^2\ar[r]^{\gamma^2}&\cdots\ar[r]^{\gamma^d}&g^{d+1}&\Sigma_dg^0\\
a^0\ar[r]^{\alpha^0}\ar[ur]^{\varphi^0}&a^1\ar[r]^{\alpha^1}\ar[ur]^{\varphi^1}&a^2\ar[r]^{\alpha^2}\ar[ur]^{\varphi^2}&\cdots\ar[r]^{\alpha^d}\ar[ur]^{\varphi^d}&a^{d+1}\ar[r]^{\alpha^{d+1}}&\Sigma_da^0}
\end{align*}

\end{setup}

\begin{lemma}\label{lemma 9 of P10} We have
\begin{enumerate}[label=\alph*)]
\item $\alpha^{i+1}\alpha^i=0$,
\item $\varphi^{i+1}\alpha^i+\gamma^{i+1}\varphi^i=0$,
\item $\gamma^{i+1}\gamma^i=0$\\
for $0\leq i\leq d-1$ and
\item $\alpha^{d+1}\alpha^d=0$
\item $(\Sigma_d\alpha^0)\alpha^{d+1}=0$
\item $(\Sigma_d\varphi^0)\alpha^{d+1}=0$
\end{enumerate}
\end{lemma}

\begin{proof}
Composing consecutive differentials in the angle \ref{set up p 8} gives zero, so 
\begin{align*}
\begin{pmatrix}
\alpha^{i+1}&0\\\varphi^{i+1}&\gamma^{i+1}
\end{pmatrix}\begin{pmatrix}
\alpha^{i}&0\\\varphi^{i}&\gamma^{i}
\end{pmatrix}=\begin{pmatrix}
\alpha^{i+1}\alpha^{i}&0\\\varphi^{i+1}\alpha^i+\gamma^{i+1}\varphi^i&\gamma^{i+1}\gamma^{i}
\end{pmatrix}
\end{align*}
is zero for $0\leq i\leq d-1$, and
\begin{align*}
\begin{pmatrix}
\alpha^{d+1}&0\\0&0
\end{pmatrix}\begin{pmatrix}
\alpha^{d}&0\\\varphi^{d}&\gamma^{d}
\end{pmatrix}=\begin{pmatrix}
\alpha^{d+1}\alpha^d&0\\0&0
\end{pmatrix}
\end{align*}
is zero, and 
\begin{align*}
\pm\begin{pmatrix}
\Sigma_d\alpha^{0}&0\\\Sigma_d\varphi^0&\Sigma_d\gamma^0
\end{pmatrix}\begin{pmatrix}
\alpha^{d+1}&0\\0&0
\end{pmatrix}=\pm\begin{pmatrix}
(\Sigma_d\alpha^0)\alpha^{d+1}\\(\Sigma_d\varphi^{0})\alpha^{d+1}
\end{pmatrix}
\end{align*}
gives zero.
\end{proof}

\begin{lemma} \label{lemma 10 of P11} There exist morphisms 
\begin{align}
\xymatrix{g^{i+1}\ar[r]^{\delta^{i+1}}&g^i} 
\end{align}
for $0\leq i\leq d$ such that
\begin{enumerate}[label=\alph*)]
\item $\gamma^d\delta^{d+1}=\operatorname{id}_{g^{d+1}}$
\item $\gamma^i\delta^{i+1}+\delta^{i+2}\gamma^{i+1}=\operatorname{id}_{g^{i+1}}$ for $0\leq i\leq d-1$
\item $\delta^1\gamma^0=\operatorname{id}_{g^0}$.
\end{enumerate}
\end{lemma}

\begin{proof}
\begin{enumerate}[label=\alph*)]
\item We can lift as follows.
\begin{align*}
\xymatrix@R=-1.75pc @C=0pc @! { && g^{d+1}\ar[dd]|{\tiny\begin{pmatrix}
0\\1
\end{pmatrix}}\ar[ddrr]^-{\tiny\begin{matrix}0\end{matrix}}\ar@{.>}[ddll]_-{\tiny\begin{pmatrix}
0\\\delta^{d+1}
\end{pmatrix}} && \\\\
a^d\oplus g^d\ar[rr]_-{\tiny\begin{pmatrix}
\alpha^{d}&0\\\varphi^{d}&\gamma^{d}
\end{pmatrix}}&&a^{d+1}\oplus g^{d+1}\ar[rr]_-{\tiny\begin{pmatrix}
\alpha^{d+1}&0\\0&0
\end{pmatrix}}&&\Sigma_d a^0\oplus\Sigma_d g^0}
\end{align*}
Since $\Hom(g^{d+1},a^d)=0$, the first component of $\begin{pmatrix}
0\\\delta^{d+1}
\end{pmatrix}$ must be zero. Hence $\gamma^d\delta^{d+1}=\operatorname{id}_{g^{d+1}}$.
\item For $i=d-1$, we can lift as follows.
\begin{align*}
\xymatrix@R=-1.75pc @C=1pc @!{ && g^{d}\ar[dd]|-{\tiny\begin{pmatrix}
0\\1-\delta^{d+1}\gamma^d
\end{pmatrix}}\ar[ddrr]^-{\tiny\begin{matrix}0\end{matrix}}\ar@{.>}[ddll]_{\tiny\begin{pmatrix}
0\\\delta^{d}
\end{pmatrix}} && \\\\
a^{d-1}\oplus g^{d-1}\ar[rr]_-{\tiny\begin{pmatrix}
\alpha^{d-1}&0\\\varphi^{d-1}&\gamma^{d-1}
\end{pmatrix}}&&a^{d}\oplus g^{d}\ar[rr]_-{\tiny\begin{pmatrix}
\alpha^{d}&0\\\varphi^d&\gamma^d
\end{pmatrix}}&&a^{d+1}\oplus g^{d+1}}
\end{align*}
Since $\Hom(g^d,a^{d-1})=0$, the first component of $\begin{pmatrix}
0\\\delta^d
\end{pmatrix}$ must be zero.  We used that
\begin{align*}
\begin{pmatrix}
\alpha^d&0\\\varphi^d&\gamma^d
\end{pmatrix}\begin{pmatrix}
0\\1-\delta^{d+1}\gamma^{d}
\end{pmatrix}=\begin{pmatrix}
0\\ \gamma^d-\gamma^d\delta^{d+1}\gamma^d
\end{pmatrix}=\begin{pmatrix}
0\\\gamma^d-\gamma^d
\end{pmatrix}=0,
\end{align*}
since $\gamma^d\delta^{d+1}=\operatorname{id}_{g^{d+1}}$ by the previous case. For $0\leq i\leq d-2$, we can lift as follows by descending induction.
\begin{align*}
\xymatrix@R=-1.75pc @C=1pc @!{ && g^{i+1}\ar[dd]|-{\tiny\begin{pmatrix}
0\\1-\delta^{i+2}\gamma^{i+1}
\end{pmatrix}}\ar[ddrr]^-{\tiny\begin{matrix}0\end{matrix}}\ar@{.>}[ddll]_-{\tiny\begin{pmatrix}
0\\\delta^{i+1}
\end{pmatrix}} && \\\\
a^{i}\oplus g^{i}\ar[rr]_-{\tiny\begin{pmatrix}
\alpha^{i}&0\\\varphi^{i}&\gamma^{i}
\end{pmatrix}}&&a^{i+1}\oplus g^{i+1}\ar[rr]_-{\tiny\begin{pmatrix}
\alpha^{i+1}&0\\\varphi^{i+1}&\gamma^{i+1}
\end{pmatrix}}&&a^{i+2}\oplus g^{i+2}}.
\end{align*}
We used that the first component of $\begin{pmatrix} 0 \\ \delta^{ i+1 } \end{pmatrix}$ is zero by $\Hom(g^{i+1},a^i)=0$ and
\begin{align*}
\begin{pmatrix}
\alpha^{i+1}&0\\\varphi^{i+1}&\gamma^{i+1}
\end{pmatrix}\begin{pmatrix}
0\\1-\delta^{i+2}\gamma^{i+1}
\end{pmatrix}&=\begin{pmatrix}
0\\ \gamma^{i+1}-(\gamma^{i+1}\delta^{i+2})\gamma^{i+1}
\end{pmatrix}\\&=\begin{pmatrix}
0\\\gamma^{i+1}-(1-\delta^{i+3}\gamma^{i+2})\gamma^{i+1} 
\end{pmatrix}\text{by induction}&\\ &=\begin{pmatrix}
0\\\gamma^{i+1}-\gamma^{i+1}
\end{pmatrix} \text{by Lemma}\, \ref{lemma 9 of P10}(c)\\
&=0.
\end{align*}
\item We can lift as follows.
\begin{align*}
\xymatrix@R=-4.5pc @C=-3pc @!{ &&& g^{0}\ar[dd]|-{\tiny\begin{pmatrix}
0\\1-\delta^{1}\gamma^{0}
\end{pmatrix}}\ar[ddrr]^-{\tiny\begin{matrix}0\end{matrix}}\ar@{.>}[ddlll]_-{\tiny\begin{pmatrix}
0\\\ast
\end{pmatrix}} && \\\\
\Sigma^{-1}_da^{d+1}\oplus\Sigma^{-1}_d g^{d+1}\ar[rrr]_-{\tiny\begin{pmatrix}
\pm\Sigma^{-1}_d\alpha^{d+1}&0\\0&0
\end{pmatrix}}&&&a^{0}\oplus g^{0}\ar[rr]_-{\tiny\begin{pmatrix}
\alpha^{0}&0\\\varphi^{0}&\gamma^{0}
\end{pmatrix}}&&a^{1}\oplus g^{1}}
\end{align*}
where we used that the first component of $\begin{pmatrix} 0 \\ \ast \end{pmatrix}$ is zero by $\Hom(g^{0},\Sigma^{-1}_da^{d+1})=0$ and
\begin{align*}
\begin{pmatrix}
\alpha^{0}&0\\\varphi^{0}&\gamma^{0}
\end{pmatrix}\begin{pmatrix}
0\\1-\delta^{1}\gamma^{0}
\end{pmatrix}&=\begin{pmatrix}
0\\ \gamma^{0}-(\gamma^{0}\delta^{1})\gamma^{0}
\end{pmatrix}\\&=\begin{pmatrix}
0\\\gamma^{0}-(1-\delta^{2}\gamma^{1})\gamma^{0}) 
\end{pmatrix}\text{by induction}&\\ &=\begin{pmatrix}
0\\\gamma^{0}-\gamma^{0}
\end{pmatrix} \text{by Lemma}\, \ref{lemma 9 of P10}(c)\\
&=0. \qedhere
\end{align*}
\end{enumerate}
\end{proof}

\begin{lemma}\label{lemma 11 P14} There exists the following biproduct diagram in the category of $(d+2)$-$\Sigma_d$-sequences. 
\begin{align*}
\xymatrixcolsep{1.3pc}\xymatrixrowsep{4pc}\xymatrix{g^0\ar@<1ex>[dd]^-{\tiny\begin{pmatrix} 0\\1\end{pmatrix}}\ar[rr]^-{\gamma^{0}}_*+[o][F-][d]{1}&&g^1\ar@<1ex>[dd]^-{\tiny\begin{pmatrix} 0\\1\end{pmatrix}}\ar[rr]^-{\gamma^{1}}&&\cdots\ar[rr]^-{\gamma^{d-1}}&&g^d\ar@<1ex>[dd]^-{\tiny\begin{pmatrix} 0\\1\end{pmatrix}}\ar[rr]^-{\gamma^{d}}_*+[o][F-][d]{2}&&g^{d+1}\ar@<1ex>[dd]^-{\tiny\begin{pmatrix} 0\\1\end{pmatrix}}\ar[rr]^0_*+[o][F-][d]{3}&&\Sigma_dg^0\ar@<1ex>[dd]^-{\tiny\begin{pmatrix} 0\\1\end{pmatrix}}\\\\
a^{0}\oplus g^{0}\ar@<1ex>[dd]^-{\tiny\begin{pmatrix} 1\,\,0\end{pmatrix}}\ar@<1ex>[uu]^-{\tiny\begin{pmatrix}
\delta^{1}\varphi^{0} \,\, 1
\end{pmatrix}}\ar[rr]^-{\tiny\begin{pmatrix}
\alpha^{0} & 0\\\varphi^{0}&\gamma^{0}
\end{pmatrix}}&&a^{1}\oplus g^{1}\ar[rr]^-{\tiny\begin{pmatrix}
\alpha^{1} & 0\\\varphi^{1}&\gamma^{1}
\end{pmatrix}}\ar@<1ex>[dd]^-{\tiny\begin{pmatrix} 1\,\,0\end{pmatrix}}\ar@<1ex>[uu]^-{\tiny\begin{pmatrix}
\delta^{2}\varphi^{1} \,\, 1 \end{pmatrix}}&&\cdots\ar[rr]^-{\tiny\begin{pmatrix}
\alpha^{d-1} & 0\\\varphi^{d-1}&\gamma^{d-1}
\end{pmatrix}}&&a^{d}\oplus g^{d}\ar[rr]^-{\tiny\begin{pmatrix}
\alpha^{d} & 0\\\varphi^{d}&\gamma^{d}
\end{pmatrix}}\ar@<1ex>[dd]^-{\tiny\begin{pmatrix} 1\,\,0\end{pmatrix}}\ar@<1ex>[uu]^-{\tiny\begin{pmatrix}\delta^{d+1}\varphi^{d}\,\,1\end{pmatrix}}&&a^{d+1}\oplus g^{d+1}\ar[rr]^{\tiny\begin{pmatrix}
\alpha^{d+1} & 0\\ 0&0
\end{pmatrix}}\ar@<1ex>[dd]^-{\tiny\begin{pmatrix} 1\,\,0\end{pmatrix}}\ar@<1ex>[uu]^-{\tiny\begin{pmatrix}0\,\,1\end{pmatrix}}&&\Sigma_da^{0}\oplus\Sigma_d g^{0}\ar@<1ex>[dd]^-{\tiny\begin{pmatrix} 1\,\,0\end{pmatrix}}\ar@<1ex>[uu]^-{\tiny\begin{pmatrix}\Sigma_d(\delta^{1}\varphi^{0})\,\,1\end{pmatrix}}\\\\
a^{0}\ar[rr]^*+[o][F-][d]{4}_{\alpha^0}\ar@<1ex>[uu]^-{\tiny\begin{pmatrix}1\\ -\delta^{1}\varphi^{0}\end{pmatrix}}&&a^{1}\ar[rr]_-{\alpha^1}\ar@<1ex>[uu]^-{\tiny\begin{pmatrix}1\\ -\delta^{2}\varphi^{1}\end{pmatrix}}&&\cdots\ar[rr]_-{\alpha^{d-1}}&&a^{d}\ar[rr]^*+[o][F-][d]{5}_-{\alpha^{d}}\ar@<1ex>[uu]^-{\tiny\begin{pmatrix}1\\ -\delta^{d+1}\varphi^{d}\end{pmatrix}}&&a^{d+1}\ar[rr]^*+[o][F-][d]{6}_-{\alpha^{d+1}}\ar@<1ex>[uu]^-{\tiny\begin{pmatrix}1\\ 0\end{pmatrix}}&&\Sigma_da^0\ar@<1ex>[uu]^-{\tiny\begin{pmatrix}1\\-\Sigma_d(\delta^{1}\varphi^{0})\end{pmatrix}}}
\end{align*}
Consequently, there are the following $(d+2)$-angles.
\begin{align*}
\xymatrix{a^{0}\ar[rr]^{\alpha^{0}}&&a^{1}\ar[rr]^{\alpha^{1}}&&\cdots\ar[rr]^{\alpha^{d-1}}&&a^{d}\ar[rr]^{\alpha^{d}}&&a^{d+1}\ar[r]^{\alpha^{d+1}}&\Sigma_da^0\\g^0\ar[rr]^{\gamma^{0}}&&g^1\ar[rr]^{\gamma^{1}}&&\cdots\ar[rr]^{\gamma^{d-1}}&&g^d\ar[rr]^{\gamma^{d}}&&g^{d+1}\ar[r]^0&\Sigma_dg^0}
\end{align*}
\end{lemma}


\begin{proof}
We start with the biproduct equations. These are clear in degree $d+1$.  In degree $i$ for $0\leq i\leq d$:
\begin{align*}
\begin{pmatrix}
\delta^{i+1}\varphi^i&1
\end{pmatrix}
\begin{pmatrix}
0\\1
\end{pmatrix}=1,
\;\;
\begin{pmatrix}
1&0
\end{pmatrix}
\begin{pmatrix}
1\\-\delta^{i+1}\varphi^i
\end{pmatrix}=1,
\end{align*}
\begin{align*}
\begin{pmatrix}
0\\1
\end{pmatrix}\begin{pmatrix}
\delta^{i+1}\varphi^i&1
\end{pmatrix}+\begin{pmatrix}
1\\-\delta^{i+1}\varphi^i
\end{pmatrix}\begin{pmatrix}
1&0
\end{pmatrix}=\begin{pmatrix}
0&0\\\delta^{i+1}\varphi^i&1
\end{pmatrix}+\begin{pmatrix}
1 &0 \\ -\delta^{i+1}\varphi^i&0
\end{pmatrix}=\begin{pmatrix}
1&0\\0&1
\end{pmatrix}.
\end{align*}
In degree $d+2$, i.e.\ the last column, the equations follow since this column is $\Sigma_d$ of the first column.\\
Now we show commutativity of the numbered parts of the diagram.
\begin{enumerate}[label=\protect\circled{\arabic*}]

\item Suppose $0\leq i\leq d-1$.
\begin{align*}
\begin{pmatrix}
\alpha^i & 0\\ \varphi^i&\gamma^i
\end{pmatrix}\begin{pmatrix}
0\\1
\end{pmatrix}=\begin{pmatrix}
0\\\gamma^i
\end{pmatrix},\quad 
\begin{pmatrix}
0\\1
\end{pmatrix}\gamma^i=\begin{pmatrix}
0\\\gamma^i
\end{pmatrix};
\end{align*}
\begin{align*}
\begin{pmatrix}
\delta^{i+2}\varphi^{i+1}&1
\end{pmatrix}\begin{pmatrix}
\alpha^i&0\\\varphi^i& \gamma^i
\end{pmatrix}&=\begin{pmatrix}
\delta^{i+2}\varphi^{i+1}\alpha^i+\varphi^i&\gamma^i
\end{pmatrix}\\
&=\begin{pmatrix}-\delta^{i+2}\gamma^{i+1}\varphi^i+\varphi^i &\gamma^i\end{pmatrix} \quad\text{Lemma}\, \ref{lemma 9 of P10}(b) \\
&=\begin{pmatrix}\left(1-\delta^{i+2}\gamma^{i+1}\right)\varphi^i&\gamma^i\end{pmatrix}\\
&=\begin{pmatrix}\gamma^i\delta^{i+1}\varphi^i&\gamma^i\end{pmatrix}\quad \text{by Lemma}\, \ref{lemma 10 of P11}(b)\\
&=\gamma^i\begin{pmatrix}\delta^{i+1}\varphi^i&1\end{pmatrix}.
\end{align*}

\item \begin{align}
\begin{pmatrix}
\alpha^d&0\\\varphi^d&\gamma^d
\end{pmatrix}\begin{pmatrix}
0\\1
\end{pmatrix}=\begin{pmatrix}
0\\\gamma^d
\end{pmatrix},\quad
\begin{pmatrix}
0\\1
\end{pmatrix}\gamma^d=\begin{pmatrix}
0\\\gamma^d
\end{pmatrix};
\end{align}
\begin{align*}
\gamma^d\begin{pmatrix}\delta^{d+1}\varphi^d&1\end{pmatrix}&=\begin{pmatrix}\gamma^d\delta^{d+1}\varphi^d&\gamma^d\end{pmatrix}\\ &=\begin{pmatrix}\varphi^d&\gamma^d\end{pmatrix}\quad\text{by Lemma}\,\ref{lemma 10 of P11}(a)\\
&=\begin{pmatrix} 0 & 1\end{pmatrix}\begin{pmatrix}
\alpha^d&0\\ \varphi^d& \gamma^d
\end{pmatrix}.
\end{align*}

\item \begin{align*}
\begin{pmatrix}
\alpha^{d+1}&0\\0&0
\end{pmatrix}\begin{pmatrix} 0\\1\end{pmatrix}=0,\quad \begin{pmatrix} 0\\1\end{pmatrix} 0=0;
\end{align*}
\begin{align*}
\begin{pmatrix} \Sigma_d(\delta^1\varphi^0)&1\end{pmatrix}\begin{pmatrix}\alpha^{d+1}&0\\0&0\end{pmatrix}=\begin{pmatrix}\Sigma_d(\delta^1\varphi^0)\alpha^{d+1}&0\end{pmatrix}=0=0\begin{pmatrix} 0&1\end{pmatrix}
\end{align*}
where we used Lemma \ref{lemma 9 of P10}(f).

\item Suppose $0\leq i\leq d-1$.
\begin{align*}
\begin{pmatrix}
1&0
\end{pmatrix}\begin{pmatrix}
\alpha^i&0\\\varphi^i&\gamma^i
\end{pmatrix}=\begin{pmatrix}
\alpha^i&0
\end{pmatrix}, \quad
\alpha^i\begin{pmatrix}
1&0
\end{pmatrix}=\begin{pmatrix}
\alpha^i&0
\end{pmatrix};
\end{align*}
\begin{align*}
\begin{pmatrix}
\alpha^i&0\\\varphi^i&\gamma^i
\end{pmatrix}\begin{pmatrix}
1\\ -\delta^{i+1}\varphi^i
\end{pmatrix}=\begin{pmatrix}
\alpha^i\\\varphi^i-\gamma^i\delta^{i+1}\varphi^i
\end{pmatrix}=\begin{pmatrix}
\alpha^i\\ (1-\gamma^i\delta^{i+1})\varphi^i
\end{pmatrix}=\begin{pmatrix}
\alpha^i\\\delta^{i+2}\gamma^{i+1}\varphi^i
\end{pmatrix}
\end{align*}
where we used Lemma \ref{lemma 10 of P11}(b) for the last equality. On the other hand, 
\begin{align*}
\begin{pmatrix}
1\\-\delta^{i+2}\varphi^{i+1}
\end{pmatrix}\alpha^i=\begin{pmatrix}
\alpha^i\\-\delta^{i+2}\varphi^{i+1}\alpha^i
\end{pmatrix}=\begin{pmatrix}
\alpha^i\\\delta^{i+2}\gamma^{i+1}\varphi^i
\end{pmatrix}\quad\text{by Lemma}\, \ref{lemma 9 of P10}(b)
\end{align*}
which proves the commutativity.

\item \begin{align*}
\begin{pmatrix}
1&0
\end{pmatrix}\begin{pmatrix}
\alpha^d&0\\\varphi^d&\gamma^d
\end{pmatrix}=\begin{pmatrix}
\alpha^d&0
\end{pmatrix},\quad
\alpha^d\begin{pmatrix}
1&0
\end{pmatrix}=\begin{pmatrix}
\alpha^d&0
\end{pmatrix};
\end{align*}
\begin{align*}
\begin{pmatrix}
\alpha^d&0\\\varphi^d&\gamma^d
\end{pmatrix}\begin{pmatrix}
1\\-\delta^{d+1}\varphi^d
\end{pmatrix}=\begin{pmatrix}
\alpha^d\\\varphi^d-\gamma^d\delta^{d+1}\varphi^d
\end{pmatrix}=\begin{pmatrix}
\alpha^d\\\varphi^d-\varphi^d
\end{pmatrix}=\begin{pmatrix}
\alpha^d\\0
\end{pmatrix}
\end{align*}
where we used Lemma \ref{lemma 10 of P11}\,a. On the other hand
\begin{align*}
\begin{pmatrix}
1\\0
\end{pmatrix}\alpha^d=\begin{pmatrix}
\alpha^d\\0
\end{pmatrix}
\end{align*}
which proves the commutativity.

\item \begin{align*}
\begin{pmatrix}
1&0
\end{pmatrix}\begin{pmatrix}
\alpha^{d+1}&0\\0&0
\end{pmatrix}=\begin{pmatrix}
\alpha^{d+1}&0
\end{pmatrix}, \quad \alpha^{d+1}\begin{pmatrix}
1&0
\end{pmatrix}=\begin{pmatrix}
\alpha^{d+1}&0
\end{pmatrix};
\end{align*}
\begin{align*}
\begin{pmatrix}
\alpha^{d+1}&0\\0&0
\end{pmatrix}\begin{pmatrix}
1\\0
\end{pmatrix}=\begin{pmatrix}
\alpha^{d+1}\\0
\end{pmatrix},\quad \begin{pmatrix}
1\\-\Sigma_d(\delta^1\varphi^0)
\end{pmatrix}\alpha^{d+1}=\begin{pmatrix}
\alpha^{d+1}\\0
\end{pmatrix}.
\end{align*}
\end{enumerate}
\end{proof}

We state the dual to Lemma \ref{lemma 11 P14}.
\begin{lemma}\label{lemma 12 of P20} In a $(d+2)$-angulated category with $d$-suspension $\Sigma^d$, consider the following $(d+2)$-angle.
\begin{align*}
\xymatrix{\Sigma^{-1}_db_0\oplus\Sigma^{-1}_dh_0\ar[rr]^{\tiny \begin{pmatrix}
\Sigma^{-1}_d\beta_0&0\\0&0
\end{pmatrix}}&&b_{d+1}\oplus h_{d+1}\ar[rr]^{\tiny \begin{pmatrix}
\beta_{d+1}&\psi_{d+1}\\0&\eta_{d+1}
\end{pmatrix}}&&b_d\oplus h_d\ar[rr]^{\tiny \begin{pmatrix}
\beta_{d}&\psi_{d}\\0&\eta_{d}
\end{pmatrix}}&&\cdots\ar[rr]^{\tiny \begin{pmatrix}
\beta_{1}&\psi_{1}\\0&\eta_{1}
\end{pmatrix}}&& b_0\oplus h_0}
\end{align*}
Assume
\begin{align*}
\Hom(b_i,h_{i+1})=0\quad\text{for}\quad i=0,\ldots,d\quad \text{and}\quad \Hom(\Sigma_db_{d+1},h_0)=0.
\end{align*}
Then the $(d+2)$-angle splits as a direct sum of the following $(d+2)$-angles:
\begin{align*}
\xymatrix{\Sigma^{-1}_db_0\ar[r]^{\Sigma^{-1}_d\beta_0}&b_{d+1}\ar[r]^{\beta_{d+1}}&b_d\ar[r]^{\beta_d}&\cdots\ar[r]^{\beta_1}&b_0,\\
\Sigma^{-1}_dh_0\ar[r]^0&h_{d+1}\ar[r]^{\eta_{d+1}}&h_d\ar[r]^{\eta_d}&\cdots\ar[r]^{\eta_1}&h_0.}
\end{align*}
\end{lemma}
Note that since the connecting morphism in the latter angle is zero, we can find a null homotopy for the latter angle consisting of $\epsilon_i:h_i\rightarrow h_{i+1}$ for $i=0,\ldots,d$ such that
\begin{align*}
\eta_1\epsilon_0=\operatorname{id}_{h_0},\,\, \eta_{i+1}\epsilon_i+\epsilon_{i-1}\eta_{i}=\operatorname{id}_{h_i}\quad\text{for}\, i=1,\ldots,d,\;\; \epsilon_d\eta_{d+1}=\operatorname{id}_{h_{d+1}}.
\end{align*}

\begin{lemma}\label{lemma 13 of P22} In an additive category, assume that 
\begin{align*}
\xymatrix{a_1\oplus a_0\ar[rr]^{\begin{pmatrix}
\alpha_{11}&\alpha_{10}\\0&\alpha_{00}
\end{pmatrix}}&&a_1\oplus a_0}
\end{align*}
is an idempotent. Then,
\begin{enumerate}[label=\alph*)]
\item $\alpha_{11}$ and $\alpha_{00}$ are idempotents.
\item If $\alpha_{11}$ and $\alpha_{00}$ are split idempotents, then so is $\alpha=\begin{pmatrix}
\alpha_{11}&\alpha_{10}\\0&\alpha_{00}
\end{pmatrix}$.
\end{enumerate}
\end{lemma}

\begin{proof}
\begin{enumerate}[label=\alph*)]
\item We have 
\begin{align*}
\begin{pmatrix}
\alpha_{11}&\alpha_{10}\\0&\alpha_{00}
\end{pmatrix}=&\begin{pmatrix}
\alpha_{11}&\alpha_{10}\\0&\alpha_{00}
\end{pmatrix}\begin{pmatrix}
\alpha_{11}&\alpha_{10}\\0&\alpha_{00}
\end{pmatrix}\\
=&\begin{pmatrix}
\alpha^2_{11}&\alpha_{11}\alpha_{10}+\alpha_{10}\alpha_{00}\\0&\alpha^2_{00}
\end{pmatrix}
\end{align*}
proving part (a). Note that we also showed 
\begin{align}\label{eq P22 Lemma 13}
\alpha_{10}=\alpha_{11}\alpha_{10}+\alpha_{10}\alpha_{00}.
\end{align}
\item Assume that $\alpha_{11}$ and $\alpha_{00}$ are split idempotents. Then we have the following biproduct diagrams splitting $\alpha_{ii}$ for $i=0,1$.
\begin{align*}
\xymatrix{b_i\ar@<1ex>[r]^{\iota_i}&a_i\ar@<1ex>[r]^{\rho_i}\ar@<1ex>[l]^{\pi_i}&c_i\ar@<1ex>[l]^{\kappa_i}}
\end{align*}
\begin{align}
\label{eq P23 13.2}\iota_i\pi_i+\kappa_i\rho_i=\operatorname{id}_{a_i}\\
\label{eq P23 13.3}\pi_i\iota_i=\operatorname{id}_{b_i}\\
\label{eq P23 13.4}\rho_i\kappa_i=\operatorname{id}_{c_i}\\
\label{eq P23 13.5}\iota_i\pi_i=\alpha_{ii}\\
\label{eq P23 13.6}\kappa_i\rho_i=\operatorname{id}_{a_i}-\alpha_{ii}
\end{align}
\end{enumerate}
Observe that 
\begin{align}
\rho_1\alpha_{10}\kappa_0=&\rho_1\left(\alpha_{11}\alpha_{10}+\alpha_{10}\alpha_{00}\right)\kappa_0\quad \text{by}\,\eqref{eq P22 Lemma 13}\nonumber \\
=&\rho_1\left(\left(\operatorname{id}_{a_1}-\kappa_1\rho_1\right)\alpha_{10}+\alpha_{10}\left(\operatorname{id}_{a_0}-\kappa_0\rho_0\right)\right)\kappa_0 \quad\text{by}\,\eqref{eq P23 13.6}\nonumber \\
=&\rho_1\left(\alpha_{10}-\kappa_1\rho_1\alpha_{10}+\alpha_{10}-\alpha_{10}\kappa_{0}\rho_0\right)\kappa_0\nonumber \\
=&2\rho_1\alpha_{10}\kappa_0-\left(\rho_1\kappa_1\rho_1\alpha_{10}\kappa_0+\rho_1\alpha_{10}\kappa_0\rho_0\kappa_0\right) \quad\text{by}\,\eqref{eq P23 13.4}\nonumber \\
=&0. \nonumber
\end{align}
Therefore 
\begin{align}
\label{eq P23 13.7} \rho_1\alpha_{10}\kappa_0=0.
\end{align}
Moreover
\begin{align}
\pi_1\alpha_{10}\iota_0=&\pi_1\left(\alpha_{11}\alpha_{10}+\alpha_{10}\alpha_{00}\right)\iota_0\nonumber\\
=&\pi_1\left(\iota_1\pi_1\alpha_{10}+\alpha_{10}\iota_0\pi_0\right)\iota_0\quad\text{by}\,\eqref{eq P23 13.5}\nonumber\\
=&(\pi_1\iota_1)\pi_1\alpha_{10}\iota_0+\pi_1\alpha_{10}\iota_0(\pi_{0}\iota_0)\nonumber\\
=&\pi_1\alpha_{10}\iota_0+\pi_1\alpha_{10}\iota_0\quad\text{by}\,\eqref{eq P23 13.3},\nonumber
\end{align}
which implies
\begin{align}\label{eq P24 13.8}
\pi_1\alpha_{10}\iota_0=0.
\end{align}

We now write down a biproduct diagram splitting $\alpha$.
\begin{align*}
\xymatrix{b_1\oplus b_0\ar@<1ex>[rr]^{\tiny \begin{pmatrix}
\iota_1 & \alpha_{10}\iota_0\\0&\iota_0
\end{pmatrix}}&&a_1\oplus a_0\ar@<1ex>[rr]^{\tiny \begin{pmatrix}
-\rho_1 & \rho_1\alpha_{10}\\0&-\rho_0
\end{pmatrix}}\ar@<1ex>[ll]^{\tiny \begin{pmatrix}
\pi_1 & \pi_1\alpha_{10}\\0&\pi_0
\end{pmatrix}}&&c_1\oplus c_0\ar@<1ex>[ll]^{\tiny \begin{pmatrix}
-\kappa_1 & \alpha_{10}\kappa_0\\0&-\kappa_0
\end{pmatrix}}}
\end{align*}
We check the required properties.
\begin{align*}
&\begin{pmatrix}
\iota_1&\alpha_{10}\iota_0\\0&\iota_0
\end{pmatrix}\begin{pmatrix}
\pi_1&\pi_1\alpha_{10}\\0&\pi_0
\end{pmatrix}+\begin{pmatrix}
-\kappa_1&\alpha_{10}\kappa_0\\0&-\kappa_0
\end{pmatrix}\begin{pmatrix}
-\rho_1&\rho_1\alpha_{10}\\0&-\rho_0
\end{pmatrix}=\\
&\begin{pmatrix}
\iota_1\pi_1&\iota_1\pi_1\alpha_{10}+\alpha_{10}\iota_0\pi_0\\0&\iota_0\pi_0
\end{pmatrix}+\begin{pmatrix}
\kappa_1\rho_1&-\kappa_1\rho_1\alpha_{10}-\alpha_{10}\kappa_0\rho_0\\0&\kappa_0\rho_0
\end{pmatrix}=\quad \text{by}\,\eqref{eq P23 13.5}\,\text{and}\,\eqref{eq P23 13.6}\\
&\begin{pmatrix}
\alpha_{11}&\alpha_{11}\alpha_{10}\\0&\alpha_{00}
\end{pmatrix}+
\begin{pmatrix}
\operatorname{id}_{a_1}-\alpha_{11}&(\alpha_{11}-\operatorname{id}_{a_1})\alpha_{10}+\alpha_{10}(\alpha_{00}-\operatorname{id}_{a_0})\\0&\operatorname{id}_{a_0}-\alpha_{00}\end{pmatrix}\\
&\begin{pmatrix}
\operatorname{id}_{a_1}&2(\alpha_{11}\alpha_{10}+\alpha_{10}\alpha_{00})-2\alpha_{10}\\0&\operatorname{id}_{a_0}\end{pmatrix}= \quad\text{by}\,\eqref{eq P22 Lemma 13}\\
&\begin{pmatrix}
\operatorname{id}_{a_1}&0\\0&\operatorname{id}_{a_0}
\end{pmatrix};
\end{align*}
\begin{align*}
&\begin{pmatrix}
\pi_1&\pi_1\alpha_{10}\\0&\pi_0
\end{pmatrix}\begin{pmatrix}
\iota_1&\alpha_{10}\iota_0\\0&\iota_0
\end{pmatrix}=
\begin{pmatrix}
\pi_1\iota_1&\pi_1\alpha_{10}\iota_0+\pi_1\alpha_{10}\iota_0\\0&\pi_0\iota_0
\end{pmatrix}=
\begin{pmatrix}
\operatorname{id}_{b_1}&0\\0&\operatorname{id}_{b_0}
\end{pmatrix}
\end{align*}
by \eqref{eq P23 13.3} and \eqref{eq P24 13.8};
\begin{align*}
&\begin{pmatrix}
-\rho_1&\rho_1\alpha_{10}\\0&-\rho_0
\end{pmatrix}\begin{pmatrix}
-\kappa_1&\alpha_{10}\\0&-\kappa_0
\end{pmatrix}=\begin{pmatrix}
\rho_1\kappa_1&-\rho_1\alpha_{10}\kappa_0-\rho_1\alpha_{10}\kappa_0\\0&\rho_0\kappa_0
\end{pmatrix}=\begin{pmatrix}
\operatorname{id}_{c_1}&0\\0&\operatorname{id}_{c_0}
\end{pmatrix}
\end{align*}
by \eqref{eq P23 13.4} and \eqref{eq P23 13.7};
\begin{align*}
&\begin{pmatrix}
\iota_1&\alpha_{10}\iota_1\\0&\iota_0
\end{pmatrix}\begin{pmatrix}
\pi_1&\pi_1\alpha_{10}\\0&\pi_0
\end{pmatrix}=\begin{pmatrix}
\iota_1\pi_1&\iota_1\pi_1\alpha_{10}+\alpha_{10}\iota_0\pi_0\\0&\iota_0\pi_0
\end{pmatrix}= \quad\text{by}\,\eqref{eq P23 13.5}\\
&\begin{pmatrix}
\alpha_{11}&\alpha_{11}\alpha_{10}+\alpha_{10}\alpha_{00}\\0&\alpha_{00}\end{pmatrix}= \quad\text{by}\,\eqref{eq P22 Lemma 13}\\
&\begin{pmatrix}
\alpha_{11}&\alpha_{10}\\0&\alpha_{00}
\end{pmatrix}=\alpha. \qedhere
\end{align*}
\end{proof}

\begin{lemma} \label{lemma 14 P27} If each $\fa_m$ has split idempotents, then so does each $\cT_{[i,s]}$.
\end{lemma}

\begin{proof}
Let $i$ be fixed. We use induction on $n=s-i$.
Set $n=0$. Then $s=i$ and $\cT_{[i,s]}=\cT_{[i,i]}=\fa_i$ has split idempotents by assumption.\\
Consider an object $a$ of $\cT_{[i,s]}$ which by definition can be written $a_s\oplus a_{s'}$ with $a_s\in \fa_s$, $a_{s'}\in\cT_{[i,s-1]}$. Let $\alpha$ be an idempotent endomorphism of $a$. By set-up \ref{setup 1 page 1} we can write
\begin{align}
\alpha=\begin{pmatrix}
\alpha_{ss}&\alpha_{ss'}\\0&\alpha_{s's'}
\end{pmatrix}.
\end{align}
Lemma \ref{lemma 13 of P22}(a) implies that $\alpha_{ss}$ and $\alpha_{s's'}$ are idempotent. By assumption, $a_{ss}$ is split idempotent and by induction $\alpha_{s's'}$ is split. But then $\alpha$ is split by Lemma \ref{lemma 13 of P22} b.
 \end{proof}
\subsection{From higher angulated categories to $d$-abelian categories}\label{subsect set up 3 proof of thm C D}
\begin{setup}\label{set up 15 P28}
We keep set-up \ref{setup 1 page 1} and assume further that:
\begin{enumerate}[label=\alph*)]
\item\label{set up 15 a} $\cT$ is $(d+2)$-angulated with $d$-suspension $\Sigma_d$.
\item\label{set up 15 b} $\fa_0$ has split idempotents.
\item\label{set up 15 c} $\fa_j=\Sigma^j_{d}\fa_0$.
\item\label{set up 15 d} $m\geq 0$ is a fixed integer and $\cT_{[0,m+1]}$ is closed under $d$-extensions in the sense: Given a morphism 
\begin{align*}
\xymatrix{t^{d+1}\ar[r]^{\delta}&\Sigma_dt^0}
\end{align*}
with $t^{d+1},t^0\in\cT_{[0,m+1]}$, there is a $(d+2)$-triangle 
\begin{align*}
\xymatrix{t^0\ar[r]&t^1\ar[r]&\cdots\ar[r]&t^d\ar[r]&t^{d+1}\ar[r]^{\delta}&\Sigma_dt^0}
\end{align*}
in $\cT$ with $t^1,\ldots,t^d\in\cT_{[0,m+1]}$.
\end{enumerate}
\end{setup}
Here we restate Theorem \ref{thm higher case B most general}.
\begin{theorem}\label{THM 16 P29 } $\cT_{[0,m]}$ is a $n$-abelian category where $n=\left(m+1\right)\left(d+2\right)-2$.
\end{theorem}
\begin{proof}
We will verify the axioms \ref{def d-abelian}.\\
\ref{d-ab axiom 0}. $\fa_0$ has split idempotents by assumption. It follows that so does each $\fa_j=\Sigma^j_d\fa_0$. Hence so does $\cT_{[0,m]}$ by Lemma \ref{lemma 14 P27}.\\
\ref{d-ab axiom 1} for $n$-cokernels. Let $\varphi:t^0\rightarrow t^1$ be a morphism in $\cT_{[0,m]}$. Writing it as $\Sigma^{-1}_d\Sigma_dt^0\rightarrow t^1$ we have $\Sigma_dt^0\in\cT_{[1,m+1]}$ so by set-up \ref{set up 15 P28}(d) there is a $(d+2)$-angle
\begin{align}\label{eq 16.1 P29}
\xymatrix{t^0\ar[r]^{\varphi}&t^1\ar[r]&t^2\ar[r]&\cdots\ar[r]&t^{d+1}\ar[r]&\Sigma_dt^0}
\end{align}
with $t^2,\ldots,t^{d+1}\in\cT_{[0,m+1]}$. It can be spun out to a diagram
\begin{align*}
\xymatrix{t^0\ar[r]&\cdots\ar[r]&t^{d+1}\ar[r]&\Sigma_dt^0\ar[r]&\cdots\ar[r]&\Sigma_dt^{d+1}\ar[r]&\cdots}
\end{align*}
which becomes exact under $\Hom_{ \cT }(-,t)$ for each $t\in\cT$. Hence
\begin{align*}
\xymatrix{\tau_{\leq m}t^0\ar[r]&\cdots\ar[r]&\tau_{\leq m}t^{d+1}\ar[r]&\tau_{\leq m}\Sigma_dt^0\ar[r]&\cdots\ar[r]&\tau_{\leq m}\Sigma_dt^{d+1}\ar[r]&\cdots}
\end{align*}
becomes exact under $\Hom_{ \cT }(-,t)$ for each $t\in\cT_{\leq m}$. In particular it becomes exact under $\Hom_{ \cT }(-,t)$ for each $t\in\cT_{[0,m]}$.
Now observe that $\xymatrix{\tau_{\leq m}t^0\ar[r]^{\tau_{\leq m}\varphi}&\tau_{\leq m}t^1}$ is simply $\xymatrix{t^0\ar[r]^{\varphi}&t^1}$ and that $\Sigma^{m+1}_dt^0\in\cT_{[m+1,2m+1]}$ whence $\tau_{\leq m}\Sigma^{m+1}_{d}t^0=0$. Hence
\begin{gather}
\xymatrix{t^0\ar[r]^{\varphi}&t^1\ar[r]&\tau_{\leq m}t^2\ar[r]&\cdots\ar[r]&\tau_{\leq m}t^{d+1}\ar[r]&}\nonumber\\
\xymatrix{\tau_{\leq m}\Sigma_dt^0\ar[r]&\cdots\ar[r]&\tau_{\leq m}\Sigma_dt^{d+1}\ar[r]&}\nonumber\\
\xymatrix{\cdots\ar[r]&}\nonumber\\
\xymatrix{\tau_{\leq m}\Sigma^m_dt^0\ar[r]&\cdots\ar[r]&\tau_{\leq m}\Sigma^m_dt^{d+1}}\label{eq 16.2}
\end{gather}
provides a $(d+m(d+2))$-cokernel of $\varphi$ in $\cT_{[0,m]}$, but we have $n=d+m(d+2)$.\\
\ref{d-ab axiom 1} for $n$-kernels: Let $\varphi:t_1\rightarrow t_0$ be a morphism in $\cT_{[0,m]}$. Writing it as $\varphi:t_1\rightarrow\Sigma_d\Sigma^{-1}_dt_0$, we have $\Sigma^{-1}_dt_0\in\cT_{[-1,m-1]}$ so by set-up \ref{set up 15 P28}(d) there is a $(d+2)$-triangle
\begin{align}\label{eq 16.3 P30}
\xymatrix{\Sigma^{-1}_dt_0\ar[r]&t_{d+1}\ar[r]&\cdots\ar[r]&t_1\ar[r]^{\varphi}\ar[r]&\Sigma_d\Sigma^{-1}_dt_0}
\end{align}
with $t_{d+1},\ldots,t_2\in\cT_{[-1,m]}$. It can be spun out to a diagram
\begin{align*}
\xymatrix{\cdots\ar[r]&\Sigma^{-1}_dt_{d+1}\ar[r]&\cdots\ar[r]&\Sigma^{-1}_dt_0\ar[r]&t_{d+1}\ar[r]&\cdots\ar[r]&t_0}
\end{align*}
which becomes an exact sequence under $\Hom_{ \cT }(t,-)$ for each $t\in\cT$. Hence
\begin{align*}
\xymatrix{\cdots\ar[r]&\tau_{\geq 0}\Sigma^{-1}_dt_{d+1}\ar[r]&\cdots\ar[r]&\tau_{\geq 0}\Sigma^{-1}_dt_0\ar[r]&\tau_{\geq 0}t_{d+1}\ar[r]&\cdots\ar[r]&\tau_{\geq 0}t_0}
\end{align*}
becomes exact under $\Hom_{ \cT }(t,-)$ for each $t\in\cT_{\geq 0}$ by Lemma \ref{lemma 6 p6}. In particular, it becomes exact under $\Hom_{ \cT }(t,-)$ for each $t\in\cT_{[0,m]}$.\\
Now observe that $\xymatrix{\tau_{\geq 0}t_1\ar[r]^{\tau_{\geq 0}\varphi}&\tau_{\geq 0}t_0}$ is simply $\xymatrix{t_1\ar[r]^{\varphi}&t_0}$ and that $\Sigma^{-m-1}_dt_0\in\cT_{[-m-1,-1]}$ whence $\tau_{\geq 0}\Sigma^{-m-1}_dt_0=0$. Hence
\begin{gather}
\xymatrix{\tau_{\geq 0}\Sigma^{-m}_dt_{d+1}\ar[r]&\cdots\ar[r]&\tau_{\geq 0}\Sigma^{-m}_dt_0\ar[r]&}\nonumber\\
\xymatrix{\tau_{\geq 0}\Sigma^{-m+1}_dt_{d+1}\ar[r]&\cdots\ar[r]&\tau_{\geq 0}\Sigma^{-m+1}_dt_0\ar[r]&}\nonumber\\
\xymatrix{\cdots\ar[r]&}\nonumber\\
\xymatrix{\tau_{\geq 0}t_{d+1}\ar[r]&\cdots\ar[r]&\tau_{\geq 0}t_2\ar[r]&t_1\ar[r]^{\varphi}&t_0}\label{eq 16.4 P30}
\end{gather}
provides a $(d+m(d+2))$-kernel of $\varphi$ in $\cT_{[0,m]}$, but we have $n=d+m(d+2)$.\\
\ref{d-ab axiom 2}. In the $(d+2)$-angle \eqref{eq 16.1 P29} assume that $\varphi:t^0\rightarrow t^1$ is monic in $\cT_{[0,m]}$. We have $t^i\in\cT_{[0,m+1]}$ for each $i$, so we can write $t^i=a^i\oplus g^i$ with $a^i\in \fa_0$, $g^i\in\cT_{[1,m+1]}$. Set-up \ref{setup 1 page 1} then implies that \eqref{eq 16.1 P29} has the form
\begin{align}\label{eq 16.5 P32}
\xymatrixcolsep{1.5pc}\xymatrix{a^{0}\oplus g^{0}\ar[rr]^-{\tiny\begin{pmatrix}
\alpha^{0}&0\\\varphi^{0}&\gamma^{0}
\end{pmatrix}}&&a^{1}\oplus g^{1}\ar[rr]^-{\tiny\begin{pmatrix}
\alpha^{1}&0\\\varphi^{1}&\gamma^{1}
\end{pmatrix}}&&a^{2}\oplus g^{2}\ar[rr]^-{\tiny\begin{pmatrix}
\alpha^{2}&0\\\varphi^{2}&\gamma^{2}
\end{pmatrix}}&&\cdots\ar[rr]^-{\tiny\begin{pmatrix}
\alpha^{d}&0\\\varphi^{d}&\gamma^{d}
\end{pmatrix}}&&a^{d+1}\oplus g^{d+1}\ar[rr]^-{\tiny\begin{pmatrix}
\alpha^{d+1}&\psi\\0&\chi
\end{pmatrix}}&&\Sigma_da^{0}\oplus \Sigma_dg^{0}}
\end{align}
where $\varphi=\tiny\begin{pmatrix}
\alpha^0&0\\\varphi^0&\gamma^0
\end{pmatrix}$.
Set-up \ref{setup 1 page 1} also implies
\begin{align*}
\Hom_{ \cT }(g^{i+1},a^i)=0\quad\text{for}\,i=0,\ldots,d,\quad \Hom_{ \cT }(g^0,\Sigma^{-1}_da^{d+1})=0.
\end{align*}
Finally, in
\begin{align*}
\xymatrixcolsep{2pc}\xymatrixrowsep{2pc}\xymatrix{\Sigma^{-1}_dg^{d+1}\ar[ddrrrr]^-{\tiny\begin{pmatrix}
\Sigma^{-1}_d\psi\\\Sigma^{-1}_d\chi
\end{pmatrix}}\ar[dd]_-{\tiny\begin{pmatrix}
0\\1
\end{pmatrix}}&&&&&&&\\\\
\Sigma^{-1}_da^{d+1}\oplus\Sigma^{-1}_dg^{d+1}\ar[rrrr]_-{\tiny\begin{pmatrix}
\Sigma^{-1}_d\alpha^{d+1}&\Sigma^{-1}_d\psi\\0&\Sigma^{-1}_d\chi
\end{pmatrix}}&&&&a^0\oplus g^0\ar[rrr]^-{\tiny\begin{pmatrix}
\alpha^0&0\\\varphi^0&\gamma^0
\end{pmatrix}}\ar@{=}[dd]&&&a^1\oplus g^1\ar@{=}[dd]\\\\
&&&&t^0\ar[rrr]_-{\tiny\begin{matrix}\varphi\end{matrix}}&&&t^1}
\end{align*}
the horizontal composition is zero so the diagram shows
\begin{align*}
\varphi\circ\begin{pmatrix}
\Sigma^{-1}_d\psi\\\Sigma^{-1}_d\chi
\end{pmatrix}=0.
\end{align*}
Since $\Sigma^{-1}_dg^{d+1}\in\cT_{[0,m]}$, this shows $\small{\begin{pmatrix}
\Sigma^{-1}_d\psi\\\Sigma^{-1}_d\chi
\end{pmatrix}}=0$ since $\varphi$ is monic in $\cT_{[0,m]}$, so the connecting morphism in \eqref{eq 16.5 P32} is in fact $\begin{pmatrix}
\alpha^{d+1}&0\\0&0
\end{pmatrix}$.\\
So \eqref{eq 16.5 P32} falls under Set-up \ref{set up p 8} whence Lemma \ref{lemma 11 P14} says that it is a direct sum of the $(d+2)$-angles
\begin{align*}
\xymatrix{a^{0}\ar[rr]^{\alpha^{0}}&&a^{1}\ar[rr]^{\alpha^{1}}&&\cdots\ar[rr]^{\alpha^{d-1}}&&a^{d}\ar[rr]^{\alpha^{d}}&&a^{d+1}\ar[r]^{\alpha^{d+1}}&\Sigma_da^0, \\
g^0\ar[rr]^{\gamma^{0}}&&g^1\ar[rr]^{\gamma^{1}}&&\cdots\ar[rr]^{\gamma^{d-1}}&&g^d\ar[rr]^{\gamma^{d}}&&g^{d+1}\ar[r]^0&\Sigma_dg^0,}
\end{align*}
where the latter is null homotopic in the sense of Lemma \ref{lemma 10 of P11}. Hence \eqref{eq 16.2} is the sum of two diagrams.
The first summand is 
\begin{gather}
\xymatrix{\tau_{\leq m}a^0\ar[r]&\tau_{\leq m}a^{1}\ar[r]&\cdots\ar[r]&\tau_{\leq m}a^{d+1}\ar[r]&}\nonumber\\
\xymatrix{\tau_{\leq m}\Sigma_da^0\ar[r]&\tau_{\leq m}\Sigma_da^1\ar[r]&\cdots\ar[r]&\tau_{\leq m}\Sigma_da^{d+1}\ar[r]&}\nonumber\\
\xymatrix{\cdots\ar[r]&}\nonumber\\
\xymatrix{\tau_{\leq m}\Sigma^m_da^0\ar[r]&\cdots\ar[r]&\tau_{\leq m}\Sigma^m_da^{d+1}}
\end{gather}
which by $a^i\in \fa_0$ can also be written
\begin{gather}
\xymatrix{a^0\ar[r]&a^{1}\ar[r]&\cdots\ar[r]&a^{d+1}\ar[r]&}\nonumber\\
\xymatrix{\Sigma_da^0\ar[r]&\Sigma_da^1\ar[r]&\cdots\ar[r]&\Sigma_da^{d+1}\ar[r]&}\nonumber\\
\xymatrix{\cdots\ar[r]&}\nonumber\\
\xymatrix{\Sigma^m_da^0\ar[r]&\cdots\ar[r]&\Sigma^m_da^{d+1}.}
\end{gather}
The second summand is a concatenation by zero morphisms of pieces which are null homotopic, since this property is preserved by the additive functor $\tau_{\leq m}$.\\
Both summands are sent to exact sequences by $\Hom_{ \cT }(t,-)$ for $t\in\cT_{[0,m]}$. For the first summand, this is clear. For the second summand, it follows since the additive functor $\Hom_{ \cT }(t,-)$ preserves the property of being a concatenation by zero morphisms of pieces which are null homotopic.\\
So \eqref{eq 16.2} is sent to an exact sequence by $\Hom_{ \cT }(t,-)$ for $t\in\cT_{[0,m]}$. Moreover, the first morphism of \eqref{eq 16.2} is the $\cT_{[0,m]}$-monic $\varphi:t^0\rightarrow t^1$; it is sent to a monomorphism by $\Hom_{ \cT }(t,-)$ for $t\in\cT_{[0,m]}$. This proves that \eqref{eq 16.2} provides an $n$-kernel of its last morphism. We already knew from [\ref{d-ab axiom 1} for $n$-cokernels] that it provides an $n$-cokernel of its first morphism. Hence \eqref{eq 16.2} is $n$-exact proving \ref{d-ab axiom 2}.\\
\ref{d-ab axiom 2 op}. In the $(d+2)$-angle \eqref{eq 16.3 P30} assume that $\varphi:t_1\rightarrow t_0$ is epic in $\cT_{[0,m]}$. We have $t_i\in\cT_{[-1,m]}$ for each $i$, so we can write $t_i=b_i\oplus h_i$ with $b_i\in\Sigma^m_d A_0$, $h_i\in\cT_{[-1,m-1]}$. Set-up \ref{setup 1 page 1} then implies that \eqref{eq 16.3 P30} has the form
\begin{align}\label{eq 16.6 P36}
\xymatrix{\Sigma^{-1}_db_0\oplus\Sigma^{-1}_dh_0 \ar[rr]^-{\tiny\begin{pmatrix}
\Sigma^{-1}_d\beta_0&0\\\theta&\chi
\end{pmatrix}}&&b_{d+1}\oplus h_{d+1}\ar[rr]^-{\tiny\begin{pmatrix}
\beta_{d+1}&\psi_{d+1}\\0&\eta_{d+1}
\end{pmatrix}}&&b_d\oplus h_d\ar[rr]^-{\tiny\begin{pmatrix}
\beta_d&\psi_d\\0&\eta_d
\end{pmatrix}}&&\cdots\ar[rr]^-{\tiny\begin{pmatrix}
\beta_{1}&\psi_{1}\\0&\eta_{1}
\end{pmatrix}}&&b_0\oplus h_0}
\end{align}
where $\varphi=\begin{pmatrix}
\beta_{1}&\psi_{1}\\0&\eta_{1}
\end{pmatrix} $.
Set-up \ref{setup 1 page 1} also implies
\begin{align*}
\Hom_{ \cT }(b_i,h_{i+1})=0\quad\text{for}\,i=0,\ldots,d,\quad \Hom_{ \cT }(\Sigma_db_{d+1},h_0)=0.
\end{align*}
Finally, in
\begin{align*}
\xymatrixcolsep{2pc}\xymatrixrowsep{2pc}
\xymatrix{t_1\ar[rrr]^-{\tiny\begin{matrix}\varphi\end{matrix}}\ar@{=}[dd]&&&t_0\ar@{=}[dd]&&&&\\\\
b_1\oplus h_1\ar[rrr]^-{\tiny\begin{pmatrix}
\beta_{1}&\psi_{1}\\0&\eta_{1}
\end{pmatrix} }&&&b_0\oplus h_0\ar[rrrr]^-{\tiny\begin{pmatrix}
-\beta_{0}&0\\-\Sigma_d\theta&-\Sigma_d\chi
\end{pmatrix} }\ar[ddrrrr]|-{\tiny\begin{pmatrix}
-\Sigma_d\theta&-\Sigma_d\chi
\end{pmatrix} }&&&&\Sigma_db_{d+1}\oplus\Sigma_dh_{d+1}\ar[dd]^-{\tiny\begin{pmatrix}
0&1
\end{pmatrix} }\\\\
&&&&&&&\Sigma_dh_{d+1}}
\end{align*}
the horizontal composition is zero so the diagram shows
\begin{align*}
\begin{pmatrix}
\Sigma^{-1}_d\theta&\Sigma^{-1}_d\chi
\end{pmatrix}\circ\varphi=0.
\end{align*}
Since $\Sigma_dh_{d+1}\in\cT_{[0,m]}$, this shows $\small{\begin{pmatrix}
\Sigma^{-1}_d\theta&\Sigma^{-1}_d\chi
\end{pmatrix}}=0$ since $\varphi$ is epic in $\cT_{[0,m]}$, so the connecting morphism in \eqref{eq 16.6 P36} is in fact $\begin{pmatrix}
\Sigma^{-1}_d\beta_0&0\\0&0
\end{pmatrix}$.\\
Hence Lemma \ref{lemma 12 of P20} says that \eqref{eq 16.6 P36} is the direct sum of the $(d+2)$-angles
\begin{align*}
\xymatrix{\Sigma^{-1}_db_0\ar[r]^{\Sigma^{-1}_d\beta_0}&b_{d+1}\ar[r]^{\beta_{d+1}}&b_d\ar[r]^{\beta_d}&\cdots\ar[r]^{\beta_1}&b_0,\\
\Sigma^{-1}_dh_0\ar[r]^0&h_{d+1}\ar[r]^{\eta_{d+1}}&h_d\ar[r]^{\eta_d}&\cdots\ar[r]^{\eta_1}&h_0}
\end{align*}
where the latter is null homotopic in the sense stated after Lemma \ref{lemma 12 of P20}. Hence \eqref{eq 16.4 P30} is the sum of two diagrams.
The first summand is 

\begin{gather*}
\xymatrix{\tau_{\geq 0}\Sigma^{-m}_db_{d+1}\ar[r]&\cdots\ar[r]&\tau_{\geq 0}\Sigma^{-m}_db_0\ar[r]&}\nonumber\\
\xymatrix{\tau_{\geq 0}\Sigma^{-m+1}_db_{d+1}\ar[r]&\cdots\ar[r]&\tau_{\geq 0}\Sigma^{-m+1}_db_0\ar[r]&}\nonumber\\
\xymatrix{\cdots\ar[r]&}\\
\xymatrix{\tau_{\geq 0}b_{d+1}\ar[r]&\cdots\ar[r]&\tau_{\geq 0}b_0,}
\end{gather*}
which by $b_i\in\Sigma^m_d\fa_0$ can also be written
\begin{gather*}
\xymatrix{\Sigma^{-m}_db_{d+1}\ar[r]&\cdots\ar[r]&\Sigma^{-m}_db_0\ar[r]&}\nonumber\\
\xymatrix{\Sigma^{-m+1}_db_{d+1}\ar[r]&\cdots\ar[r]&\Sigma^{-m+1}_db_0\ar[r]&}\nonumber\\
\xymatrix{\cdots\ar[r]&}\\
\xymatrix{b_{d+1}\ar[r]&\cdots\ar[r]&b_0.}
\end{gather*}

The second summand is a concatenation by zero morphisms of pieces which are null homotopic, since this property is preserved by the additive functor $\tau_{\geq m}$.\\
Both summands are sent to exact sequences by $\Hom_{ \cT }(-,t)$ for $t\in\cT_{[0,m]}$. For the first summand, this is clear. For the second summand, it follows since the additive functor $\Hom_{ \cT }(-,t)$ preserves the property of being a concatenation by zero morphisms of pieces which are null homotopic.\\
So \eqref{eq 16.4 P30} is sent to an exact sequence by $\Hom_{ \cT }(-,t)$ for $t\in\cT_{[0,m]}$. Moreover, the last morphism of \eqref{eq 16.4 P30} is the $\cT_{[0,m]}$-epic $\varphi:t_1\rightarrow t_0$; it is sent to a monomorphism by $\Hom_{ \cT }(-,t)$ for $t\in\cT_{[0,m]}$. This proves that \eqref{eq 16.4 P30} provides an $n$-cokernel of its first morphism. We already knew from [\ref{d-ab axiom 1} for $n$-kernels] that it provides an $n$-kernel of its last morphism. Hence \eqref{eq 16.4 P30} is $n$-exact proving \ref{d-ab axiom 2 op}.

\end{proof}

Here we restate Theorem \ref{thm higher case B, rep finite}:
\begin{theorem}
Let $\cM$ be an $d$-cluster tilting subcategory of $\modd A$ where $A$ is $d$-representation finite algebra. Then, the full subcategory of $D^b(\modd A)$ which is generated by $\bigoplus^{m}_{j=0}\cM[jd]$ is an $n$-abelian category where $n=(d+2)(m+1)-2$. 
\end{theorem}
\begin{proof}
It is enough to show that the full subcategory of $D^b(\modd A)$ which is generated by $\bigoplus^{m}_{j=0}\cM[jd]$ has the properties stated in Set-ups \ref{setup 1 page 1} and \ref{set up 15 P28} if we set $\fa_j = \cM[ jd ]$.  Set-up \ref{setup 1 page 1} and Set-up \ref{set up 15 P28}, parts \ref{set up 15 b} and \ref{set up 15 c} are clearly satisfied, and Set-up \ref{set up 15 P28}, part \ref{set up 15 a} holds by \cite[Thm 1]{geiss2013n}.  We will show in Proposition \ref{prop 6 wide p10 } that Set-up \ref{set up 15 P28}, part \ref{set up 15 d} is satisfied.  Therefore, by Theorem \ref{THM 16 P29 } the claim holds.
\end{proof}

\subsection{From abelian to $d$-abelian categories}\label{subsect set up 4 proof of thm A}

\begin{setup}\label{setup 17 P40} Let $\cC$ be an abelian category. We define the (bounded) derived category as in \cite[Def. 13.1.2]{KS06}, observing that $D^b(\cC)$ is indeed the full subcategory of $D(\cC)$ of complexes with bounded homology by \cite[Prop 13.1.12(i)]{KS06}.\\
For $c',c''\in\cC$ we define
\begin{align*}
\Ext^i_{\cC}(c'',c')=\Hom_{D(\cC)}(c'',\Sigma^ic')=\Hom_{D^b(\cC)}(c'',\Sigma^ic'')
\end{align*}
as in \cite[p.322 not. 3.1.9]{KS06}.\\
Assume that $\cC$ is hereditary in the sense of \cite[p.\ 324, def 13.1.18]{KS06}. Set $\cT=D^b(\cC)$, $\fa_m=\Sigma^{m}\cC$. Note that by \cite[cor 13.1.20]{KS06},
\begin{align}\label{eq 17.1 p41}
\cT_{[i,s]}=\left\{t\in D^b(\cC)\,\vert\,H_{\ast}t\,\text{is concentrated in degrees}\, i,\ldots,s\right\}.
\end{align}
\end{setup}
Now all assumptions in Set-ups \ref{setup 1 page 1} and \ref{set up 15 P28} are satisfied with $d=1$:
\begin{enumerate}[label=\arabic*)]
\item \ref{setup 1 page 1} \ref{set up 1 a} By definition of hereditary category.
\item \ref{setup 1 page 1} \ref{set up 1 b}  By \eqref{eq 17.1 p41}.
\item \ref{set up 15 P28} \ref{set up 15 a} By \cite[p.\ 319]{KS06}
\item \ref{set up 15 P28} \ref{set up 15 b} Well known for abelian categories.
\item \ref{set up 15 P28} \ref{set up 15 c} By definition.
\item \ref{set up 15 P28} \ref{set up 15 d} By \eqref{eq 17.1 p41} and the long exact sequence.
\end{enumerate}
\begin{theorem}
If $\cC$ is a hereditary abelian category, then
\begin{align*}
\left\{t\in D^b(\cC)\,\vert\,H_{\ast}t\;\text{is concentrated in degrees}\; 0,\ldots,m\right\}
\end{align*}
is a $(3m+1)$-abelian category.
\end{theorem}
\begin{proof}
This follows from the combination of Set-up \ref{setup 17 P40} and Theorem \ref{THM 16 P29 }.
\end{proof}

\subsection{Examples}\label{subsect set up 5 proof of Corollary}

We want to prove each statement in Corollary \ref{thm MAIN COR}. We restate Corollary \ref{thm MAIN COR}:
\begin{corollary}
i) Let $\cC$ be the category of finite dimensional modules over a finite dimensional hereditary algebra which is not of finite representation type. Then, $\cC_{[0,m]}$ is a $(3m+1)$-abelian category with infinitely many indecomposable objects.\vspace{0.1cm}\\
ii) Let $\cC$ be the category of abelian groups. Then, $\cC_{[0,m]}$ is a $(3m+1)$-abelian category which is not $\mathbb{K}$-linear over a field $\mathbb{K}$ but has set indexed products and coproducts. It is $\mathbb{Z}$-linear and it does not have a Serre functor.\vspace{0.1cm}\\
iii) Let $\cC$ be the category of coherent sheaves over the projective line. Then, $\cC_{[0,m]}$ is a $(3m+1)$-abelian category with not enough injectives.\vspace{0.1cm}\\
iv) Let $\cC'$ be a category derived equivalent to a hereditary abelian category $\cC$, that is there is an equivalence of triangulated categories $\cF : D^b(\cC) \xrightarrow{} D^b(\cC')$.  Then, $\cF(\cC_{[0,m]})$ is a $(3m+1)$-abelian category.
\end{corollary}

\begin{proof} In each part, Theorem \ref{thm hereditary implies d-abelian} implies that $\cC_{[0,m]}$ is $(3m+1)$-abelian.
i) Since $\cC$ has infinitely many indecomposable objects, $\cC_{[0,m]}$ does as well.\\
ii) In this case, $\cC\cong\Modd\mathbb{Z}$. Since $\Modd\mathbb{Z}$ is a $\mathbb{Z}$-linear category having both products and coproducts, $(\Modd\mathbb{Z})_{[0,m]}$ has the same properties. It is clear that $(\Modd\mathbb{Z})_{[0,m]}$ is not $\mathbb{K}$-linear, hence does not have Serre functor.\\ 
iii) The category $\cC={\rm coh}(\mathbb{P}^1)$ of coherent sheaves on $\mathbb{P}^1$ is a hereditary abelian category without enough injectives. An object $X\in\cC$ without an injective preenvelope in $\cC$ does not have an injective preenvelope in $\cC_{[0,m]}$.\\
iv) The functor $\cF : D^b(\cC) \xrightarrow{} D^b(\cC')$ restricts to an equivalence $\cC_{[0,m]} \xrightarrow{} \cF( \cC_{[0,m]} )$.
\end{proof}

\begin{remarks}
i) It follows from \cite[thm.\ 3.3]{EN20} that $d$-abelian categories with set indexed products and coproducts exists in $\text{Mod}$-$A$ when $A$ is a $d$-representation finite algebra. Here we obtain this property based on the category of $\mathbb{Z}$-modules, which is not $\mathbb{K}$-linear. \\
ii) We point out $\cF(\cC_{[0,m]})$ is not necessarily of the form $\cC'_{[0,m]}$.
\end{remarks}

\section{Only hereditary abelian categories work}\label{section d-abelian to actual category}
We provide proof of theorem \ref{thm d-abelian implies hereditary A'}.
Let $\cC$ be a small abelian category, $m\geq 0$ an integer. We set $n=3m+1$. 
\begin{proposition}
$\cC_{[0,m]}$ satisfies Definition \ref{def d-abelian}, \ref{d-ab axiom 0}, i.e., it has split idempotents.
\end{proposition}

\begin{proof}
$D^b(\cC)$ has split idempotents by  \cite[cor.\ 2.10]{BS01}. Since $\cC_{[0,m]}$ is an additive subcategory of $D^b(\cC)$, the claim follows.
\end{proof}


\begin{proposition}\label{prop 4}
$\cC_{[0,m]}$ satisfies Definition \ref{def d-abelian}, \ref{d-ab axiom 1}, i.e., it has $n$-kernels and $n$-cokernels.
\end{proposition}

\begin{proof}
Given a morphism $\xi: X\rightarrow Y$ in $\cC_{[0,m]}$, we complete it to a triangle in $D^b(\cC)$, then construct the following diagram where $\Sigma$ is the suspension functor on $D^b(\cC)$.
\begin{align}\label{sect1 eq1}
\xymatrix{X\ar[r]^{\xi}&Y\ar[r]&Z\ar[r]&\Sigma X\ar[r]&\Sigma Y\ar[r]&\Sigma Z\ar[r]&\cdots}
\end{align}
Notice that $\Sigma X\in\cC_{[1,m+1]}$, so the long exact homology sequence implies $Z\in\cC_{[0,m+1]}$. In particular $X,Y,Z\in\cC_{[0,m+1]}$, hence 
\begin{align}\label{sect1 eq2}
\Sigma^iX,\Sigma^iY,\Sigma^iZ\in\cC_{[i,i+m+1]}.
\end{align}
If $U\in D^b(\cC)$, then $\Hom_{D^b(\cC)}(-,U)$ maps \eqref{sect1 eq1} to an exact sequence. Combining with \cite[Prop. 13.1.16]{KS06}, if $V\in\cC_{[0,m]}$, then $\Hom_{D^b(\cC)}(-,V)$ maps
\begin{align}\label{sect1 eq3}
\xymatrix@C=4.5ex{\tau_{\leq m}X\ar[r]^{\tau_{\leq m}\xi}& \tau_{\leq m}Y\ar[r]&\tau_{\leq m}Z\ar[r]& \tau_{\leq m}\Sigma X\ar[r]& \tau_{\leq m}\Sigma Y\ar[r]& \tau_{\leq m}\Sigma Z\ar[r]&\cdots}
\end{align}
to an exact sequence, where $\tau_{ \leq m }$ is the soft truncation functor to homological degrees $\leq m$.  But \eqref{sect1 eq2} implies that \eqref{sect1 eq3} is in $\cC_{[0,m]}$ and has the form
\begin{align}
\xymatrix@C=2.1ex{X\ar[r]^{\xi}& Y\ar[r]& \tau_{\leq m}Z\ar[r]&\cdots\ar[r]& \tau_{\leq m}\Sigma^m X\ar[r]& \tau_{\leq m}\Sigma^m Y\ar[r]& \tau_{\leq m}\Sigma^m Z\ar[r]& 0\ar[r]& 0\ar[r]&\cdots.}
\end{align}
Hence the sequence \eqref{sect1 eq3} provides a $(3m+1)$-cokernel of $\xi:X\rightarrow Y$ in $\cC_{[0,m]}$, i.e., an $n$-cokernel.  The 
dual argument provides an $n$-kernel.
\end{proof}

\begin{proposition}\label{prop 5} Assume $m\geq 1$.
\begin{enumerate}[label=\arabic*)]
\item If $\cC$ has enough injectives and $\cC_{[0,m]}$ satisfies Definition \ref{def d-abelian},
\ref{d-ab axiom 2}, then $\cC$ is a hereditary category.
\item If $\cC$ has enough projectives and $\cC_{[0,m]}$ satisfies Definition \ref{def d-abelian}, \ref{d-ab axiom 2 op}, then $\cC$ is a hereditary category.
\end{enumerate}
\end{proposition}
 \begin{proof}
 Suppose that $\cC$ has enough injectives but is not hereditary. We show that Definition \ref{def d-abelian}, \ref{d-ab axiom 2} fails.
 
 There exists a non-trivial $2$-extension 
 \begin{align*}
 0\rightarrow C_1\rightarrow E\rightarrow F\rightarrow C_0\rightarrow 0
 \end{align*}
 in $\cC$, which is the Yoneda product of $1$-extensions
\[
  0\rightarrow C_1\rightarrow E\rightarrow C'\rightarrow 0
  ,\;\;
  0\rightarrow C'\rightarrow F\rightarrow C_0\rightarrow 0.
\]
 This corresponds to morphisms $\gamma': C'\rightarrow \Sigma C_1$ and $\gamma_0: C_0\rightarrow \Sigma C'$ in $D^b(\cC)$ such that the composition $\Sigma(\gamma')\circ\gamma_0$ is nonzero.
 
 We pick a short exact sequence $\xymatrix{0\ar[r]& C'\ar[r]^{\phi'}& I\ar[r]& C''\ar[r]& 0}$ in $\cC$ with $I$ injective. It induces a triangle $\xymatrix{C'\ar[r]^{\phi'}&I\ar[r]& C''\ar[r]^{\phi''}& \Sigma C'}$  in $D^b(\cC)$.  It is easy to see that $\phi'$ is monic in $\cC_{[0,m]}$. By \cite[Lem. 1.4.3]{Nee01}, there is a commutative diagram 
\begin{equation}
 \label{eq 1}
 \xymatrix{C'\ar[d]_{\gamma'}\ar[r]^{\phi'}&I\ar[r] \ar[d]^{i}&C''\ar[r]^{\phi''}\ar@{=}[d]&\Sigma C'\ar[d]^{\Sigma\gamma'}\\
\Sigma C_1\ar[r]_{\sigma} & X\ar[r]_{\xi}& C''\ar[r] &\Sigma C_1}
\end{equation}
 where the rows are triangles, such that there is a triangle
\begin{align}\label{eq 2}
 \xymatrix{C'\ar[rr]^-{\tiny \begin{pmatrix}
 \phi'\\\gamma'
 \end{pmatrix}}&&I\oplus\Sigma C_1\ar[rr]^-{\tiny \begin{pmatrix}
 i,-\sigma
 \end{pmatrix}}&&X\ar[r]^{\phi''\xi}&\Sigma C'.}
\end{align}
 The lower triangle in \eqref{eq 1} shows $X\in\cC_{[0,1]}$. Moreover, $\begin{pmatrix}
 \phi'\\\gamma'
 \end{pmatrix}$ is a morphism in $\cC_{[0,m]}$ since $m \geq 1$ and it is a monic in $\cC_{[0,m]}$ since $\phi'$ is monic in $\cC_{[0,m]}$.
 The construction in the proof of Proposition \ref{prop 4} shows that the following diagram obtained from \eqref{eq 2} provides an $n$-cokernel of $\begin{pmatrix}
 \phi'\\\gamma'\end{pmatrix}$.
\begin{align}
\xymatrix{\cdots\ar[r]&\tau_{\leq m}\Sigma^{m-1}X\ar[r]\ar@{=}[d]&\tau_{\leq m}\Sigma^{m}C'\ar[r]\ar@{=}[d]&\tau_{\leq m}\Sigma^{m}(I\oplus\Sigma C_1)\ar[r]\ar@{=}[d]&\tau_{\leq m}\Sigma^{m}X\ar@{=}[d]\\
\cdots\ar[r]&\Sigma^{m-1}X\ar[r]&\Sigma^{m}C'\ar[r]&\Sigma^{m}I\ar[r]&\tau_{\leq m}\Sigma^{m}X}
\end{align} 
The equalities follow from $C',I,C_1\in\cC$ and $X\in\cC_{[0,1]}$.
 Now, we can show that Definition \ref{def d-abelian}, \ref{d-ab axiom 2} fails. The axiom would imply that if $U\in \cC_{[0,m]}$, then $\Hom(U,-)$ sent 
 \begin{align*}
 \xymatrix{\cdots\ar[r]&\Sigma^{m-1}X\ar[r]&\Sigma^{m}C'\ar[r]&\Sigma^{m}I\ar[r]&\tau_{\leq m}\Sigma^{m}X}
 \end{align*}
 to an exact sequence. Setting $U=\Sigma^{m-1}C_0$, this would imply
 \begin{align}
 \xymatrix{\Hom(C_0,X)\ar[rr]^{(\phi''\xi)_{\ast}}&&\Hom(C_0,\Sigma C')\ar[rr]^{(\Sigma\phi')_{\ast}}&&\Hom(C_0,\Sigma I)}
 \end{align}
were exact. Since $I$ is injective, the last term is zero, so this would imply $(\phi''\xi)_{\ast}$ were surjective.
 However, \eqref{eq 2} gives an exact sequence
  \begin{align}
 \xymatrix{\Hom(C_0,X)\ar[rr]^{(\phi''\xi)_{\ast}}&&\Hom(C_0,\Sigma C')\ar[rr]^-{\tiny -\begin{pmatrix}
 \Sigma\phi'\\\Sigma\gamma'
 \end{pmatrix}_{\ast}}&&\Hom(C_0,\Sigma I\oplus\Sigma^2C_1)}
 \end{align}
 where  
 \begin{align*}
 -\begin{pmatrix}
 \Sigma\phi'\\\Sigma\gamma'
 \end{pmatrix}_{\textstyle \ast}(\gamma_0)=-\begin{pmatrix}
 \Sigma(\phi')\circ\gamma_0\\\Sigma(\gamma')\circ\gamma_0
 \end{pmatrix}\neq 0
 \end{align*}
 since $\Sigma(\gamma')\circ\gamma_0\neq 0$.  This proves $(\phi''\xi)_{\ast}$ not surjective.
 
 For the second statement, apply a dual argument.
 \end{proof}
 
 \begin{theorem}
 Assume that $m\geq 1$ and $\cC$ has enough injectives or enough projectives.  If $\cC_{[0,m]}$ is $n$-abelian then $\cC$ is a hereditary category.
 \end{theorem}
 \begin{proof}
This follows from Proposition \ref{prop 5}. 
 \end{proof}

\section{Correspondence between certain wide subcategories}\label{section wide}
We give the proof of Theorem \ref{thm wide subcats}.  The objects from Set-ups \ref{setup 1 page 1} and \ref{set up 15 P28} will be defined as follows, cf.\ Subsection \ref{subsec:Higher}.
\begin{setup}
Let $\mathbb{K}$ be a field, $\Lambda$ a finite dimensional $\mathbb{K}$-algebra with $\gldim\Lambda\leq d$.
\begin{enumerate}[label=\alph*)]
  \item  $\fa$ is a $d$-cluster tilting subcategory of $\modd\Lambda$.
  \item  $\fa_j=\Sigma^{dj}\fa$ in $D^b(\modd\Lambda)$.
  \item  $\cT=\add\left(\Sigma^{d\mathbb{Z}}A\right)$ is the corresponding $d$-cluster tilting subcategory of $D^b(\modd\Lambda)$; see \cite[Thm 1.21]{iyama2007auslander}.
\end{enumerate}
\end{setup}

Set-up \ref{setup 1 page 1} and Set-up \ref{set up 15 P28}, parts \ref{set up 15 b} and \ref{set up 15 c} are clearly satisfied, and Set-up \ref{set up 15 P28}, part \ref{set up 15 a} holds by \cite[Thm 1]{geiss2013n}.  We will show in Proposition \ref{prop 6 wide p10 } that Set-up \ref{set up 15 P28}, part \ref{set up 15 d} is satisfied.

Recall that $\tau_{\geq i}$ and $\tau_{\leq s}$ are the functors on $\cT$ from Lemmas \ref{lemma 6 p6} and \ref{lemma 7 p7}.

\begin{lemma}\label{lemma 2 wide p 2}
Suppose $x\in D^b(\modd\Lambda)$ has $H_{\ast}x$ concentrated in degree zero and let $\xi:x\rightarrow t$ be a $\cT$-envelope. Then, $H_{\ast}t$ is concentrated in degree zero and $H_{0}\xi$ is monic.
\end{lemma}
\begin{proof}
Replacing $x$ by an isomorphic object, we can assume that $x$ is an object of $\modd\Lambda$. Let $\xi:x\rightarrow a$ be an $\fa$-envelope. It is monic and $a\in\fa\subset\cT$, so it is enough to prove that each morphism $\tilde{\xi}:x\rightarrow\tilde{t}$ with $\tilde{t}\in\cT$ indecomposable factors through $\xi$, since $\xi:x\rightarrow a$ can then be used as $\xi: x\rightarrow t$ in the lemma. Observe that we have $\tilde{t}=\Sigma^{di}\tilde{a}$ for some $i \in \mathbb{Z}$ and $\tilde{a}\in\fa$ indecomposable. If $i\notin\{0,1\}$ then $\tilde{\xi}=0$ since $\gldim\Lambda\leq d$, so $\tilde{\xi}$ factors through $\xi$ trivially.
If $i=0$ then $\tilde{t}=\tilde{a}$ so $\tilde{\xi}$ factors through the $\fa$-envelope $\xi$. Now note that there is a short exact sequence 
$\xymatrix{0\ar[r]&x\ar[r]^{\xi}&a\ar[r]&c\ar[r]&0}$ in $\modd\Lambda$, yielding a triangle $\xymatrix{\Sigma^{-1}c\ar[r]&x\ar[r]^{\xi}&a\ar[r]&c}$ in $D^b(\modd\Lambda)$.
If $i=1$ then $\tilde{t}=\Sigma^d\tilde{a}$ and we consider the following diagram.
\begin{align*}
\xymatrix{\Sigma^{-1}c\ar[rrd]_{\gamma}\ar[rr]&&x\ar[d]^-{\tilde{\xi}}\ar[rr]^{\xi}&&a\ar@{..>}[lld]^{\alpha}\\
&&\Sigma^d\tilde{a}&&}
\end{align*}
The morphism $\gamma$ is zero because $c,\tilde{a}\in\modd\Lambda$ while $\gldim\Lambda\leq d$ so the factorization $\alpha$ exists.
\end{proof}

\begin{lemma}\label{lemma 3 wide p 4} Suppose $x\in D^b(\modd\Lambda)$ has $H_{\ast}x$ concentrated in degrees $\leq 0$ and let $\xi: x\rightarrow t$ be a $\cT$-envelope. Then, $H_{\ast}t$ is concentrated in degrees $\leq 0$ and $H_{0}\xi$ is monic.
\end{lemma}
\begin{proof}
(a) Replacing $x$ by an isomorphic object, we can assume $x$ is concentrated in non-positive homological degrees.  The object
$t$ is a direct sum of objects of the form $\Sigma^{di}a$ with $i\in\mathbb{Z}$, $a\in\fa$. Since $\gldim\Lambda\leq d$, each morphism $x\rightarrow \Sigma^{dj}a$ with $j\geq 2$ is zero, so $\xi$ has the form 
\begin{align*}
\xymatrix{x\ar[rr]^-{\tiny\begin{pmatrix}
\xi_a \\ \xi_u
\end{pmatrix}}&&\Sigma^da\oplus u}
\end{align*} with $a\in\fa$ and $u\in\cT$ concentrated in non-positive homological degrees.
It is enough to prove that each morphism $\tilde{\xi}:x\rightarrow \tilde{t}$ with $\tilde{t}\in\cT$ indecomposable factors through $\xi_u$. Observe that we have $\tilde{t}=\Sigma^{di}\tilde{a}$ for some $i\in\mathbb{Z}$ and $\tilde{a}\in\fa$ indecomposable.
If $i\geq 2$ then $\tilde{\xi}=0$ since $\gldim\Lambda\leq d$, so $\tilde{\xi}$ factors through $\xi_u$ trivially.

If $i\leq 0$ then we have a factorization 
\begin{align*}
\xymatrix{x\ar[d]^-{\tilde{\xi}}\ar[rr]^-{\tiny\begin{pmatrix}
\xi_a \\ \xi_u
\end{pmatrix} }&&\Sigma^da\oplus u\ar@{..>}[lld]|-{\tiny\begin{pmatrix}
\varphi_a&\varphi_u
\end{pmatrix}}\\
\tilde{t}=\Sigma^{di}\tilde{a}&&}
\end{align*}
where $\varphi_a$ represents a negative extension between modules whence $\varphi_a=0$ so $\tilde{\xi}=\begin{pmatrix}
\varphi_a&\varphi_u
\end{pmatrix}\begin{pmatrix}
\xi_a\\\xi_u
\end{pmatrix}=\varphi_u\xi_u$.\\
Assume $i=1$. Hard truncation provides a short exact sequence of chain complexes $\xymatrix{0\ar[r]&x'\ar[r]^{f'}&x\ar[r]^f&x_0\ar[r]&0}$ with $x'$ concentrated in homological degrees $\leq -1$ and $x_0\in\modd\Lambda$, hence a short triangle $\xymatrix{x'\ar[r]^{f'}&x\ar[r]^f&x_0}$ in $D^b(\modd\Lambda)$. By Lemma \ref{lemma 2 wide p 2} there is a $\cT$-envelope $\xi_0:x_0\rightarrow a_0$ with $a_0\in\fa$. We get the following diagram.
\begin{align*}
\xymatrixcolsep{3pc}\xymatrixrowsep{3pc}\xymatrix{x'\ar[rr]^{f'}&&x\ar@/_4pc/[dddd]_{\tilde{\xi}}\ar[dd]|-{\tiny\begin{pmatrix}
\xi_a\\\xi_u
\end{pmatrix}}\ar[rr]^{f}&&x_0\ar@/^6pc/[ddddll]^{\alpha}\ar[dd]^{\xi_0}\\\\
&&\Sigma^da\oplus u 
\ar[rr]^{\tiny\begin{pmatrix}
0&\psi_u
\end{pmatrix}}&&a_0\ar[ddll]^{\varphi_0}\\\\
&&\tilde{t}=\Sigma^d\tilde{a}&&}
\end{align*}
To construct the diagram, observe that since $\gldim\Lambda\leq d$ we have $\tilde{\xi}f'=0$ so there exists $\alpha:x_0\rightarrow \Sigma^d\tilde{a}$ such that $\alpha f=\tilde{\xi}$. Since $\xi_0$ is a $\cT$-envelope there exists $\varphi_0:a_0\rightarrow \Sigma^d\tilde{a}$ such that $\alpha=\varphi_0\xi_0$.\\
Since $\xi=\begin{pmatrix}
\xi_a\\\xi_u
\end{pmatrix}$ is a $\cT$-envelope, there exists $\begin{pmatrix}
0&\psi_u
\end{pmatrix}: \Sigma^da\oplus u\rightarrow a_0$
such that $\xi_0 f=\begin{pmatrix}
0&\psi_u
\end{pmatrix}\begin{pmatrix}
\xi_a\\\xi_u
\end{pmatrix}=\psi_u\xi_u$. Notice that the first component of $\begin{pmatrix}
0&\psi_u
\end{pmatrix}$ is zero since it represents a negative extension of modules. But then,
\begin{align*}
\tilde{\xi}=\alpha f=\varphi_0\xi_0 f=\varphi_0\psi_u\xi_u
\end{align*}
so $\tilde{\xi}$ has been factored through $\xi_u$ as desired.

\noindent
(b) As above, hard truncation provides a short triangle $\xymatrix{x'\ar[r]^{f'}&x\ar[r]^f&x_0}$ with $x'$ concentrated in homological degrees $\leq -1$ and $x_0\in\modd\Lambda$. The long exact homology sequence contains $\xymatrix{0=H_0x'\ar[r]&x\ar[r]^{H_0f}&x_0}$ so $H_0f$ is monic. Let $\xi_0:x_0\rightarrow t_0$ be a $\cT$-envelope. Then we get the factorization
\begin{align*}
\xymatrix{x\ar[d]^{\xi}\ar[r]^f&x_0\ar[d]^{\xi_0}\\
t\ar@{..>}[r]_{g}&t_0}
\end{align*}
Now $H_0g\circ H_0\xi=H_0\xi_0\circ H_0f$ where $H_0\xi_0$ is monic by Lemma \ref{lemma 2 wide p 2} and $H_0f$ is monic by the above.  So $H_0\xi$ is monic as claimed.
\end{proof}

\begin{lemma}\label{lemma 4 wide p 8} Suppose $x\in D^b(\modd\Lambda)$ has $H_{\ast}x$ concentrated in degrees $\geq 1$ and let $\xi:x\rightarrow t$ be a $\cT$-envelope. Then $H_{\ast}t$ is concentrated in degrees $\geq d$.
\end{lemma}
\begin{proof}
Each object in $\cT$ is the direct sum of objects $\Sigma^{di}a$ with $i\in\mathbb{Z}$ and $a\in\fa$.  If $i\leq 0$ then each morphism $x\rightarrow \Sigma^{di}a$ is zero so $\Sigma^{di}a$ is a superfluous summand in the envelope and must be zero. This implies the lemma.
\end{proof}

\begin{lemma}\label{lemma 5 wide p 9} Let $n\geq 0$ be an integer, suppose $x\in D^b(\modd\Lambda)$ satisfies
\begin{enumerate}[label=\alph*)]
\item\label{eq wide page 9} $H_*x$ is concentrated in degrees $-d+1,\ldots,nd$,
\end{enumerate}
and let $\xymatrix{x\ar[r]^{\xi}&t\ar[r]&y}$ be a short triangle with $\xi$ a $\cT$-envelope. Then $H_*y$ is concentrated in degrees $-d+2,\ldots,nd$ and, in particular, in degrees $-d+1,\ldots,nd$.
\end{lemma}
\begin{proof}
If $d=1$ then $\cT=D^b(\modd\Lambda)$ whence $\xi=\operatorname{id}_x$ and $y\cong 0$ so the lemma is vacuously true. Assume $d\geq 2$. By (de)suspending lemmas \ref{lemma 3 wide p 4} and \ref{lemma 4 wide p 8}, we learn
\begin{enumerate}[label=\alph*)]
\setcounter{enumi}{1} 
\item\label{eq wide p9 b} $H_*t$ is concentrated in degrees $0,\ldots,nd$,
\item\label{eq wide p9 c} $H_{nd}\xi$ is a monomorphism.
\end{enumerate}
The triangle gives a long exact sequence of pieces
\begin{align*}
\xymatrix{H_{i+1}y\ar[r]&H_ix\ar[r]^{H_i\xi}&H_it\ar[r]&H_iy\ar[r]&H_{i-1}x}.
\end{align*}
Hence \ref{eq wide page 9} and \ref{eq wide p9 b} imply that $H_*y$ is concentrated in degrees $-d+2,\ldots,nd+1$, and \ref{eq wide p9 b} and \ref{eq wide p9 c} imply $H_{nd+1}y=0$.
\end{proof}

\begin{proposition}\label{prop 6 wide p10 } If $n\geq 0$ is an integer, then
\begin{align*}
\cT_{[0,n]}=\add\left(\fa_0,\fa_1,\ldots,\fa_n\right)=\add\left(\fa,\Sigma^d\fa,\ldots,\Sigma^{nd}\fa\right)
\end{align*}
is closed under $d$-extensions in $\cT$ in the sense specified in Setup \ref{set up 15 P28}, part \ref{set up 15 d}.
\end{proposition}

\begin{proof}
Let $t^{d+1},t^0\in\cT_{[0,n]}$ and a morphism $\xymatrix{\Sigma^{-d}t^{d+1}\ar[r]^-{\delta}&t^0}$ be given. We construct a $(d+2)$-angle
\begin{align*}
\xymatrix{\Sigma^{-d}t^{d+1}\ar[r]^{\delta}&t^0\ar[r]^{f^0}&t^1\ar[r]^{f^1}&\cdots\ar[r]^{f^{d-1}}&t^d}
\end{align*}
by constructing a diagram of triangles in $D^b(\modd\Lambda)$
\begin{align*}
\xymatrix@C=1.3em{&&t^0\ar[rr]^{f^0}\ar[rd]&&t^1\ar[rr]^{f^1}\ar[rd]&&t^2\ar[rd]&&\cdots&&t^{d-1}\ar[rd]^{f^{d-1}}&&\\
&\Sigma^{-d}t^{d+1}\ar[ru]^{\delta}&&x^1\ar@{~>}[ll]\ar[ru]^{\xi^1}&&x^2\ar@{~>}[ll]\ar[ru]^{\xi^2}&&x^3\ar@{~>}[ll]&\cdots&x^{d-1}\ar[ru]^{\xi^{d-1}}&&x^d=t^d\ar@{~>}[ll]}
\end{align*}
where each $\xi^i$ is a $\cT$-envelope. Note that the last cone, $x^d=t^d$, is in $\cT$ because $\cT$ is $d$-cluster tilting.
The triangle $\Sigma^{-d}t^{d+1}\rightarrow t^0\rightarrow x^1\rightarrow \Sigma^{-d+1}t^{d+1}$ gives a long exact sequence showing that $H_{\ast}x^1$ is concentrated in degrees $-d+1,\ldots,nd$.
(De)suspending lemmas \ref{lemma 3 wide p 4} and \ref{lemma 4 wide p 8} shows that $H_{\ast}t^1$ is concentrated in degrees $0,\ldots,nd$ so $t^1\in\cT_{[0,n]}$. Lemma \ref{lemma 5 wide p 9} shows that $H_{\ast}x^2$ is concentrated in degrees $-d+1,\ldots,nd$.\\
Repeating the argument shows $t^2,\ldots,t^{d-1}\in\cT_{[0,n]}$ and $H_{\ast}x^3,\ldots,H_{\ast}x^d$ concentrated in degrees $-d+1,\ldots,nd$. But we know $t^d\in\cT$ so $H_{\ast}x^d=H_{\ast}t^d$ concentrated in degrees $-d+1,\ldots,nd$ implies $t^d\in\cT_{[0,n]}$.
\end{proof}

Recall that $\cT_{[0,m]}$ is $n$-abelian for $n = (d+2)(m+1) - 2$ by Theorem \ref{thm higher case B most general}.

\begin{proposition}\label{prop 7 wide p12} If $\cU$ is a wide subcategory of the $(d+2)$-angulated category $\cT$ in the sense of \cite[Def. 3.3]{Fe19}, then $\cU\cap\cT_{[0,m]}$ is a wide subcategory of the $n$-abelian category $\cT_{[0,m]}$ in the sense of \cite[Def. 2.8]{HJV20}.
\end{proposition}
\begin{proof}
$\cU\cap\cT_{[0,m]}$ is closed under $n$-cokernels in $\cT_{[0,m]}$:  Let $\xymatrix{u^0\ar[r]^{\varphi}&u^1}$ be a morphism in $\cU\cap\cT_{[0,m]}$. It can be completed to a $(d+2)$-angle in $\cT$,
\begin{align*}
\xymatrix{u^0\ar[r]^{\varphi}&u^1\ar[r]^{\varphi^1}&t^2\ar[r]^{\varphi^2}&t^3\ar[r]^{\varphi^3}&\cdots\ar[r]&t^d\ar[r]^{\varphi^d}&t^{d+1}\ar[r]^{\varphi^{d+1}}&\Sigma_du^0}.
\end{align*}
By the proof of Theorem \ref{thm higher case B most general}, \ref{d-ab axiom 1} for $n$-cokernels we can choose $t^2,\ldots,t^{d+1}\in\cT_{[0,m+1]}$. By wideness of $\cU$, we can choose $t^2,\ldots,t^{d+1}\in\cU$. Each of these completions has as a direct summand the minimal completion where $\varphi^2,\ldots,\varphi^{d}\in\rad$ (see \cite[Lem. 3.14]{Fe19}), so the minimal completion has $t^2,\ldots,t^{d+1}\in\cU\cap\cT_{[0,m+1]}$. Note that this is also true for $d=1$. If $\varphi$ is used as $\varphi$ in the proof of Theorem \ref{thm higher case B most general}, \ref{d-ab axiom 1} for $n$-cokernels, we can hence assume $t^2,\ldots,t^{d+1}\in\cU\cap\cT_{[0,m+1]}$. But then \eqref{eq 16.2} provides an $n$-cokernel of $\varphi$ with all objects in $\cU\cap\cT_{[0,m]}$ because $\tau_{\leq m}$ sends $\cU$ to $\cU$ because it sends an object $t$ to a direct summand of $t$.

$\cU\cap\cT_{[0,m]}$ is closed under $n$-kernels in $\cT_{[0,m]}$: Proved analogously using the proof of Theorem \ref{thm higher case B most general}, \ref{d-ab axiom 1} for $n$-kernels.

$\cU\cap\cT_{[0,m]}$ is closed under $n$-extensions in $\cT_{[0,m]}$: 
Let 
\begin{align}\label{eq wide 7.1 p13}
\xymatrix{t^0\ar[r]^{\varphi}&t^1\ar[r]&\cdots\ar[r]&t^n\ar[r]^{\psi}&t^{n+1}}
\end{align}
be an $n$-exact sequence in $\cT_{[0,m]}$ with $t^0,t^{n+1}\in\cU\cap\cT_{[0,m]}$. We must show that it is equivalent to an $n$-extension with $t^1,\ldots,t^n\in\cU\cap\cT_{[0,m]}$.
 We can assume each unlabelled  arrow is in $\rad$ since if not, then this can be accomplished by dropping from each unlabelled arrow all trivial summands $\xymatrix{x\ar[r]^{\sim}&x}$ (doing so does not change the equivalence class of the $n$-extension). We can also assume $\varphi,\psi\in\rad$, since if not, then this can be accomplished by the same means.  Here the end terms and $t^1$ and $t^n$ change, but the new end terms are still in $\cU\cap\cT_{[0,m]}$ and if the new $t^1$ and $t^n$ are in $\cU\cap\cT_{[0,m]}$, then so were the old ones.

By the proof of Theorem \ref{thm higher case B most general}, \ref{d-ab axiom 1} for $n$-cokernels we can complete $\varphi$ to the $(d+2)$-angle \eqref{eq 16.1 P29} with $t^2,\ldots,t^{d+1}\in\cT_{[0,m+1]}$.  Writing $t^i=a^i\oplus g^i$ with $a^i\in\fa$, $g^i\in\cT_{[1,m+1]}$ places us in Set-up \ref{set up p 8} because $\varphi$ is $\cT_{[0,m]}$-monic. Hence Lemma \ref{lemma 11 P14} applies, providing a trivial summand which we drop from \eqref{eq 16.1 P29}. This cannot affect $\varphi$ since $\varphi$ is in $\rad$, so the new $(d+2)$-angle can be used as \eqref{eq 16.1 P29}, but by Lemma \ref{lemma 11 P14} it has the form
 \begin{align}\label{eq 7.2 wide page 14}
 \xymatrix{a^0\ar@{=}[d]\ar[r]_{\alpha^0}&a^1\ar@{=}[d]\ar[r]_{\alpha^1}&\cdots\ar[r]_{\alpha^{d-1}}&a^d\ar[r]_{\alpha^{d}}&a^{d+1}\ar[r]_{\alpha^{d+1}}&\Sigma_da^0\\
 t^0\ar[r]^{\varphi}&t^1&&&&}
 \end{align}
 with $\alpha^0=\varphi$. Dropping additional trivial summands $\xymatrix{x\ar[r]^{\sim}&x}$ from $\alpha^1,\ldots,\alpha^d$ enables us to assume $\alpha^1,\ldots,\alpha^d\in\rad$, and $\alpha^{d+1}\in\rad$ since $a^{d+1}\in\fa$, $\Sigma_da^0\in\Sigma_d\fa$. So each morphism in \eqref{eq 7.2 wide page 14} is in $\rad$.\\
 Using \eqref{eq 7.2 wide page 14} as \eqref{eq 16.1 P29} provides an $n$-extension \eqref{eq 16.2} of the form
 \begin{align}\label{eq 7.3 wide page 15}
 \xymatrix@C=0.7em{a^0\ar@{=}[d]\ar[r]^{\alpha^0}&a^1\ar@{=}[d]\ar[r]&\cdots\ar[r]&a^{d+1}\ar[r]&\Sigma_da^0\ar[r]&\Sigma_da^1\ar[r]&\cdots\ar[r]&\Sigma_da^{d+1}\ar[r]&\cdots\ar[r]&\Sigma^{m}_{d}a^0\ar[r]&\Sigma^{m}_{d}a^1\ar[r]&\cdots\ar[r]&\Sigma^{m}_{d}a^{d+1}\\
 t^0\ar[r]^{\varphi}&t^1&&&&}
 \end{align}
 with each arrow in $\rad$. By uniqueness of minimal $n$-cokernels, it must be isomorphic to \eqref{eq wide 7.1 p13} whence $\Sigma^{m}_{d}a^{d+1}\cong t^{n+1}\in\cU$ so $a^{d+1}\in\cU$ whence minimality of \eqref{eq 7.2 wide page 14} implies $a^1,\ldots,a^d\in\cU$. So each object in \eqref{eq 7.3 wide page 15}, hence in \eqref{eq wide 7.1 p13}, is in $\cU\cap\cT_{[0,m]}$ as desired.
\end{proof}

\begin{definition}\label{definition of certain wide cats wide page 16} If $\cW \subseteq \fa$ is an additive subcategory, then we define additive subcategories
\begin{enumerate}[label=\alph*)]
\item $\underline{\cW}=\add\left(\cW,\Sigma_d\cW,\cdots,\Sigma^m_d\cW\right)\subseteq \cT_{[0,m]}$,
\item $\overline{\cW}:=\add\left(\Sigma^i_d\cW\vert i\in\mathbb{Z} \right)\subseteq\cT$.
\end{enumerate}
\end{definition}
By \cite[Lem. 4.2]{Fe19} there are inverse bijections
\begin{align*}
    \left\{
        \begin{aligned}
            & \text{additive subcategories} \\
            & \text{of}\,\,\fa \\
        \end{aligned}
    \right\}  &  \left.
        \begin{aligned}
            &\ni \cW\longmapsto \;\;\overline{\cW}\in \\
            &\ni \fa\cap\cU\reflectbox{\small $\longmapsto$}\; \cU\in \\
        \end{aligned}
    \right.  \left\{
        \begin{aligned}
            & \text{additive subcategories} \\
            & \text{of}\,\, \cT\,\,\text{invariant under}\,\,\Sigma^{\pm}_d \\
        \end{aligned}
    \right\}.  
\end{align*}
We have $\overline{\cW}\cap\cT_{[0,m]}=\underline{\cW}$.

\begin{proposition}\label{prop 10 wide page 17}
Let $\cW\subseteq \fa$ be an additive subcategory. Then,
\begin{align*}
\underline{\cW}\,\,\text{is wide in}\,\,\cT_{[0,m]}\implies\cW\,\,\text{is wide in}\,\,\fa.
\end{align*}
\end{proposition}
\begin{proof}
Assume that $\underline{\cW}$ is wide in $\cT_{[0,m]}$. 

\noindent
$\cW$ is closed under $d$-cokernels:  Let $\xymatrix{w^0\ar[r]^{\varphi^0}&w^1}$ be a morphism in $\cW$. We use it as the morphism $\xymatrix{t^0\ar[r]^{\varphi}&t^1}$ in the proof of Theorem \ref{thm higher case B most general}, \ref{d-ab axiom 1} for $n$-cokernels and construct the $(d+2)$-angle \eqref{eq 16.1 P29} and the diagram \eqref{eq 16.2} providing an $n$-cokernel of $\varphi$ in $\cT_{[0,m]}$. Dropping trivial direct summands of the form $\xymatrix{x\ar[r]^{\sim}&x}$ from all morphisms in \eqref{eq 16.2} apart from the first two produces a new diagram providing a minimal $n$-cokernel of $\varphi$ in $\cT_{[0,m]}$ and here all objects must be in $\underline{\cW}$ because $\underline{\cW}$ is wide in $\cT_{[0,m]}$. The first part of the new diagram is 
\begin{align*}
\xymatrix{w^0\ar[r]&w^1\ar[r]&\underline{w}^2\ar[r]&\cdots\ar[r]&\underline{w}^{d+1}\ar[r]&\Sigma_dw^0}
\end{align*}
with $\underline{w}^2,\ldots,\underline{w}^{d+1}\in\underline{\cW}$. We rewrite it as
\begin{align}\label{eq 10.1 wide page 18}
\xymatrixcolsep{3pc}\xymatrixrowsep{3pc}\xymatrix{w^{0}\ar[r]^{\varphi^0}&w^{1}\ar[r]^-{\tiny\begin{pmatrix}
\varphi^{1}\\\psi^{1}
\end{pmatrix}}&w^{2}\oplus g^2\ar[r]^-{\tiny\begin{pmatrix}
\varphi^{2}&0\\\psi^{2}&\gamma^{2}
\end{pmatrix}}&w^{3}\oplus g^3\ar[r]^-{\tiny\begin{pmatrix}
\varphi^{3}&0\\\psi^{3}&\gamma^{3}
\end{pmatrix}}&\cdots\ar[r]^-{\tiny\begin{pmatrix}
\varphi^{d}&0\\\psi^{d}&\gamma^{d}
\end{pmatrix}}&w^{d+1} \oplus g^{d+1} \ar[r]^-{\tiny\begin{pmatrix}
\omega&\delta\end{pmatrix}}&\Sigma_dw^{0}}
\end{align}
with $w^2,\ldots,w^{d+1}\in\cW$, $g^2,\ldots,g^{d+1}\in\add\left(\Sigma_d\cW,\ldots,\Sigma^{m}_{d}\cW\right)$. The zero entries of the matrices are forced since they represent negative extensions. By the $n$-cokernel property, $\cT_{[0,m]}(-,t)$ maps \eqref{eq 10.1 wide page 18} to an exact sequence for $t\in\cT_{[0,m]}$. In particular, we could let $t=w\in\cW$. But applying $\cT_{[0,m]}(-,w)$ to \eqref{eq 10.1 wide page 18} gives
\begin{align*}
\xymatrixcolsep{2pc}\xymatrixrowsep{3pc}\xymatrix{\cW(w^{0},w)&\ar[l]_{\varphi^{0\ast}}\cW(w^{1},w)&\ar[l]_{\varphi^{1\ast}}\cW(w^{2},w)&\ar[l]_{\varphi^{2\ast}}\cW(w^{3},w)&\cdots\ar[l]&\ar[l]\cW(w^{d+1},w)&\ar[l]0}.
\end{align*}
This means that 
\begin{align*}
\xymatrix{w^{0}\ar[r]^{\varphi^{0}}&w^{1}\ar[r]^{\varphi^{1}}&w^{2}\ar[r]^{\varphi^{2}}&w^{3}\ar[r]^{\varphi^{3}}&\cdots\ar[r]^{\varphi^{d}}&w^{d+1}}
\end{align*}
provides a $d$-cokernel of $\varphi^0$ consisting of objects in $\cW$.

\noindent
$\cW$ is closed under $d$-kernels: It is enough to prove that $\Sigma^{m}_{d}\cW$ is closed under $d$-kernels. Let $\xymatrix{\Sigma^{m}_{d}w_1\ar[r]^{\varphi_1}&\Sigma^{m}_{d}w_0}$ be a morphism in $\Sigma^{m}_{d}\cW$. We use it as the morphism $\xymatrix{t_1\ar[r]^{\varphi}&t_0}$ in proof of Theorem \ref{thm higher case B most general}, \eqref{d-ab axiom 1} for $n$-kernels and construct the $(d+2)$-angle \eqref{eq 16.3 P30} and the diagram \eqref{eq 16.4 P30} providing an $n$-kernel of $\varphi$ in $\cT_{[0,m]}$. Dropping trivial direct summands $\xymatrix{x\ar[r]^{\sim}&x}$ from all morphisms in \eqref{eq 16.4 P30} apart from the last two produces a new diagram providing a minimal $n$-kernel of $\varphi$, and here all objects must be in $\underline{\cW}$ because $\underline{\cW}$ is wide in $\cT_{[0,m]}$. The last 
part of the new diagram is
\begin{align*}
\xymatrix{\Sigma^{m-1}_{d}w_{0}\ar[r]&\underline{w}_{d+1}\ar[r]&\cdots\ar[r]&\underline{w}_2\ar[r]&\Sigma^{m}_{d}w_{1}\ar[r]^{\varphi_1}&\Sigma^{m}_{d}w_{0}}
\end{align*}
with $\underline{w}_2,\ldots\underline{w}_{d+1} \in \cW$. We rewrite this as
\begin{align}\label{eq 10.2 wide page 19}
\xymatrix{\Sigma^{m-1}_{d}w_{0}\ar[r]^-{\tiny\begin{pmatrix}
\omega\\\delta
\end{pmatrix}}&\Sigma^{m}_{d}w_{d+1}\oplus h_{d+1}\ar[rr]^-{\tiny\begin{pmatrix}
\varphi_{d+1}&\psi_{d+1}\\0&\eta_{d+1}
\end{pmatrix}}&&\cdots\ar[r]^-{\tiny\begin{pmatrix}
\varphi_{3}&\psi_{3}\\0&\eta_{3}
\end{pmatrix}}&\Sigma^{m}_{d}w_{2}\oplus h_2\ar[r]^-{\tiny\begin{pmatrix}
\varphi_{2}\,\,\psi_{2}
\end{pmatrix}}&\Sigma^{m}_{d}w_{1}\ar[r]^{\varphi_1}&\Sigma^{m}_{d}w_{0}}
\end{align}
with $\Sigma^{m}_{d}w_2,\ldots,\Sigma^{m}_{d}w_{d+1}\in\Sigma^{m}_{d}\cW$, $h_2,\ldots,h_{d+1}\in\add\left(\cW,\Sigma^{m}_{d}\cW,\ldots,\Sigma^{m-1}_{d}\right)$. The zero entries of the matrices are forced since they represent negative extensions. By the $n$-kernel property, $\cT_{[0,m]}(t,-)$ maps \eqref{eq 10.2 wide page 19} to an exact sequence for $t\in\cT_{[0,m]}$. In particular, we could let $t=\Sigma^{m}_{d}w\in\Sigma^{m}_{d}\cW$. But applying $\cT_{[0,m]}(\Sigma^{m}_{d}w,-)$ to \eqref{eq 10.2 wide page 19} gives 
\begin{align*}
\xymatrix{0\ar[r]&( \Sigma^{m}_{d}\cW )(\Sigma^{m}_{d}w,\Sigma^{m}_{d}w_{d+1})\ar[r]^-{\varphi_{d+1\ast}}&\cdots\\\cdots\ar[r]^-{\varphi_{3\ast}}& ( \Sigma^{m}_{d}\cW )(\Sigma^{m}_{d}w,\Sigma^{m}_{d}w_{2})\ar[r]^{\varphi_{2\ast}}& ( \Sigma^{m}_{d}\cW )(\Sigma^{m}_{d}w,\Sigma^{m}_{d}w_{1})\ar[r]^{\varphi_{1\ast}}& ( \Sigma^{m}_{d}\cW )(\Sigma^{m}_{d}w,\Sigma^{m}_{d}w_{0})}
\end{align*}
This means that 
\begin{align*}
\xymatrix{\Sigma^{m}_{d}w_{d+1}\ar[r]^-{\varphi_{d+1}}&\cdots\ar[r]^{\varphi_3}&\Sigma^{m}_{d}w_{2}\ar[r]^{\varphi_{2}}&\Sigma^{m}_{d}w_{1}\ar[r]^{\varphi_{1}}&\Sigma^{m}_{d}w_{0}}
\end{align*}
provides a $d$-kernel of $\varphi_1$ consisting of objects in $\Sigma^{m}_{d}\cW$.\\

$\cW$ is closed under $d$-extensions: Let 
\begin{align*}
\xymatrix{a^{0}\ar[r]^{\alpha^{0}}&a^{1}\ar[r]^{\alpha^{1}}&\cdots\ar[r]&a^{d}\ar[r]^{\alpha^{d}}&a^{d+1}}
\end{align*}
be a $d$-exact sequence in $\fa$ with $a^0,a^{d+1}\in\cW$ and $\alpha^1,\ldots,\alpha^{d-1}\in\rad$. We must show $a^1,\ldots,a^d\in\cW$. Dropping trivial summands $\xymatrix{x\ar[r]^{\sim}&x}$ from $\alpha^0$ and $\alpha^d$, we can assume we also have $\alpha^0,\alpha^d\in\rad$. The induced $(d+2)$-angle 
\begin{align*}
\xymatrix{a^{0}\ar[r]^{\alpha^{0}}&a^{1}\ar[r]^{\alpha^{1}}&\cdots\ar[r]&a^{d}\ar[r]^{\alpha^{d}}&a^{d+1}\ar[r]^{\alpha^{d+1}}&\Sigma_da^0}
\end{align*}
also has $\alpha^{d+1}\in\rad$ because $a^{d+1}\in\fa,\Sigma_da^0\in\Sigma_d\fa$. Using this angle as \eqref{eq 16.1 P29} gives a diagram \eqref{eq 16.2} of the form
\begin{align}\label{eq 10.3 wide page 21}
& \xymatrix{a^{0}\ar[r]^{\alpha^0}&a^{1}\ar[r]&\cdots\ar[r]&a^{d}\ar[r]&a^{d+1}\ar[r]&\Sigma_d a^{0}\ar[r]&\Sigma_d a^{1}\ar[r]&\cdots}\\
\nonumber
& \xymatrix{\cdots\ar[r]&\Sigma_d a^{d}\ar[r]&\Sigma_d a^{d+1}\ar[r]&\cdots}\\
\nonumber
& \xymatrix{\cdots\ar[r]&\Sigma^{m}_{d} a^{0}\ar[r]&\Sigma^{m}_{d} a^{1}\ar[r]&\cdots\ar[r]&\Sigma^{m}_{d} a^{d}\ar[r]&\Sigma^{m}_{d} a^{d+1}}
\end{align}
with each arrow in $\rad$. The diagram provides a cokernel of $\alpha^0$ in $\cT_{[0,m]}$. Now note that since $\xymatrix{a^0\ar[r]^{\alpha^0}&a^1}$ is monic in $\fa$, it is also monic in $\cT_{[0,m]}=\add\left(\fa,\Sigma^d\fa,\ldots,\Sigma^{md}\fa\right)$ since $\Sigma^d\fa,\ldots,\Sigma^{md}\fa$ only have zero morphisms to $a^0$ and $a^1$. So \eqref{eq 10.3 wide page 21} is a minimal $n$-exact sequence in $\cT_{[0,m]}$. Since $a^0\in\cW\subseteq\underline{\cW}$ and $\Sigma^{m}_{d}a^{d+1}\in\Sigma^{m}_{d}\cW\subseteq\underline{\cW}$ while $\underline{\cW}$ is wide, we learn $a^1,\ldots,a^d\in\underline{\cW}$ whence $a^1,\ldots,a^{d}\in\cW$ as desired.
\end{proof}

The following lemma has essentially the same proof as \cite[Lem. 4.3]{Fe19}.
\begin{lemma}\label{lemma 11 wide page 22} Let $\cW\in\fa$ be an additive subcategory. Then $\cW$ is functorially finite in $\fa$ $\Leftrightarrow$ $\underline{\cW}$ is functorially finite in $\cT_{[0,m]}$.
\end{lemma}

\begin{proof}
$\Rightarrow$: Since $\cW$ is functorially finite in $\fa$ and $\fa$ is functorially finite in $\modd\Lambda$ and $\modd\Lambda$ is functorially finite in $D^b(\modd\Lambda)$, we learn that $\cW$ is functorially finite in $D^b(\modd\Lambda)$. Since $\Sigma$ is an automorphism, $\Sigma^j\cW$ is functorially finite in $D^b(\modd\Lambda)$ for each $j$. Given $t\in\cT_{[0,m]}\subset D^b(\modd\Lambda)$,
 we pick for each $i$ a $\Sigma^{id}\cW$-precover $\xymatrix{\Sigma^{id}w^{'}_{i}\ar[r]^-{\omega^{'}_{i}}&t}$ and a $\Sigma^{id}\cW$-preenvelope $\xymatrix{t\ar[r]^-{\omega^{''}_i}&\Sigma^{id}w^{''}_i}$. Then
\begin{align*}
& \xymatrix{\Sigma^{0d}w^{'}_0\oplus\cdots\oplus\Sigma^{md}w^{'}_{m}\ar[rrrrr]^-{\omega'=\tiny\begin{pmatrix}
\omega^{'}_{0},\ldots,\omega^{'}_{m}
\end{pmatrix}}&&&&&t},\\
& \xymatrix{t\ar[rrrr]^-{\omega^{''}=\tiny\begin{pmatrix}
\omega^{'}_{0}\\\hdots\\\omega^{'}_{m}
\end{pmatrix}}&&&&\Sigma^{0d}w^{''}_0\oplus\cdots\oplus\Sigma^{md}w^{''}_{m}}
\end{align*}
 are a $\cW$-precover and a $\cW$-preenvelope.
 
 \noindent
 $\Leftarrow:$ Let $a\in\fa$ be given.\\
 Since $a\in\cT_{[0,m]}$   there is a $\cW$-cover $\xymatrix{\underline{w}^{'}\ar[r]^{\underline{\omega}^{'}}&a}$. We have $\underline{w}'=\Sigma^{0d}w^{'}_{0}\oplus\cdots\oplus\Sigma^{md}w^{'}_{d}$ with $w^{'}_{0},\ldots,w^{'}_{d}\in\cW$, and $\Sigma^{1d}w^{'}_{1},\ldots,\Sigma^{md}w^{'}_{d}$ have only zero morphisms to $a$ so must be zero. So $\underline{\omega}'$ has the form $w^{'}_{0}\rightarrow a$, and this is clearly a $\cW$-cover of $a$.\\
 
 Since $\Sigma^{md}a\in\cT_{[0,m]}$   there is a $\cW$-envelope $\xymatrix{\Sigma^{md}a\ar[r]^{\underline{\omega}^{''}}&\underline{w}^{''}}$. We have $\underline{w}''=\Sigma^{0d}w^{''}_{0}\oplus\cdots\oplus\Sigma^{md}w^{''}_{d}$ with $w^{''}_{0},\ldots,w^{''}_{d}\in\cW$, and $\Sigma^{md}a$ has only zero morphisms to $\Sigma^{0d}w^{''}_{d},\ldots,\Sigma^{(m-1)d}w^{''}_{m-1}$ which must accordingly be zero. So $\underline{\omega}''$ has the form $\Sigma^{md}a\rightarrow \Sigma^{md}w^{''}_{d}$. In particular, this has the preenveloping property with respect to $\Sigma^{md}\cW$, so $a\rightarrow w^{''}_{d}$ is a $\cW$-envelope.
 \end{proof}

\begin{definition}\label{def 12 of wide repetitive subcategories} The additive subcategories of $\cT_{[0,m]}$ of the form $\underline{\cW}$ for an additive subcategory $\cW\subseteq\fa$ will be called repetitive.
\end{definition}

\begin{theorem}
There are inverse bijections
\begin{align*}
    \left\{
        \begin{aligned}
        &\text{functorially finite wide}\\
            & \text{subcategories of}\,\,\fa
        \end{aligned}
    \right\}  &  \left.
        \begin{aligned}
            &\ni \cW\longmapsto \;\;\;\;\; \overline{\cW}\in \\
            &\ni \fa\cap\mathscr{X} \reflectbox{\small $\longmapsto$} \; \mathscr{X}\in \\
        \end{aligned}
    \right.  \left\{
        \begin{aligned}
            & \text{functorially finite repetitive} \\
            & \text{wide subcategories of}\,\, \cT_{[0,m]}\\
        \end{aligned}
    \right\}.
\end{align*}
\end{theorem}

\begin{proof}
Let $\cW\subseteq\fa$ be a functorially finite wide subcategory. Then $\underline{\cW}\subseteq\cT_{[0,m]}$ is functorially finite by Lemma \ref{lemma 11 wide page 22} and repetitive by definition \ref{def 12 of wide repetitive subcategories}. By \cite[Thm. A]{Fe19}, $\overline{\cW}$ is wide in $\cT$ whence $\overline{\cW}\cap\cT_{[0,m]}=\underline{\cW}$ is wide in $\cT_{[0,m]}$ by Proposition \ref{prop 7 wide p12}.

Let $\fX\subseteq\cT_{[0,m]}$ be a functorially finite repetitive wide subcategory. If we set $\cW=\fa\cap\fX$ then $\fX=\underline{\cW}$. Lemma \ref{lemma 11 wide page 22} gives $\cW\subseteq\fa$ functorially finite and Proposition \ref{prop 10 wide page 17} gives $\cW\subseteq\fa$ wide.

The maps are clearly inverses.
\end{proof}

\section{Questions}\label{section questions} 
 
\begin{problem} 
Corollary \ref{thm MAIN COR} provides $d$-abelian categories with interesting properties for $d=3m+1$ where $m\geq 0$ is an integer.  Let $d$ be an integer which is not congruent to $1$ modulo $3$.  
\begin{enumerate}[label=\arabic*)]
\item Does there exist a $d$-abelian category that is not $\mathbb{K}$-linear but possesses both products and coproducts?
\item Does there exist a $d$-abelian category that does not have a Serre functor?
\item Does there exist a $d$-abelian category that does not have enough injectives?
\end{enumerate}
\end{problem}

\begin{problem}
In Theorem \ref{thm d-abelian implies hereditary A'} we show the reverse implication of Theorem \ref{thm hereditary implies d-abelian}.  Is there an analogous reverse implication of Theorem \ref{thm higher case B most general}?
\end{problem}

\begin{problem}\label{problem 6.3}
Let $\Lambda$ be a finite dimensional algebra, $\modd\Lambda$ the category of finitely generated left modules, $D^b(\modd\Lambda)$ its derived category. For which values of $d$ does there exist a $d$-abelian category $\cC$ contained in $D^b(\modd\Lambda)$ ?
\end{problem}

\bibliographystyle{alpha}

\end{document}